\documentclass[11pt]{article}

 \usepackage{amsmath,amssymb,amscd,amsthm,esint}
 
\usepackage{graphics,amsmath,amssymb,amsthm,mathrsfs}

\usepackage{graphics,amsmath,amssymb,amsthm,mathrsfs,amsfonts}

\usepackage{sidecap}
\usepackage{float}
\usepackage{extarrows}
\usepackage{booktabs}
\usepackage{verbatim}
\usepackage{hyperref}
\usepackage[usenames,dvipsnames]{xcolor}

\newtheorem{thm}{Theorem}[section]
\newtheorem{lemma}[thm]{Lemma}

\theoremstyle{definition}
\newtheorem{remark}[thm]{Remark}

\def\XXint#1#2#3{{\setbox0=\hbox{$#1{#2#3}{\int}$}
         \vcenter{\hbox{$#2#3$}}\kern-.5\wd0}}

\def\R{\mathbb{R}}

\def\e{\varepsilon}

\def\tY{\widetilde{Y}}

\numberwithin{equation}{section}

\begin{document}

\title{ Uniform Estimates for Dirichlet Problems \\ in Perforated Domains}

\author{
Zhongwei Shen\thanks{Supported in part by NSF grants DMS-1856235, DMS-2153585,  and by Simons Fellowship.}    
}
\date{}
\maketitle

\begin{abstract}

 This paper studies  the Dirichlet problem for Laplace's equation in a domain $\Omega_{\varepsilon, \eta}$ perforated with small holes,
 where $\varepsilon$ represents the scale of the minimal distances between holes and $\eta$ the ratio between the scale of  sizes of
 holes and $\varepsilon$.
 We establish  $W^{1, p}$ estimates for solutions with bounding constants depending explicitly on the small parameters $\varepsilon$ and $\eta$.
  We also show that these estimates are either optimal or near optimal.
  
\medskip


\medskip


\end{abstract}


\section{Introduction}\label{section-1}

In this paper we consider the Dirichlet problem for Laplace's equation,
\begin{equation}\label{D-01}
\left\{
\aligned
-\Delta u & =F+ \text{\rm div} (f) & \quad &  \text{ in } \Omega_{\e, \eta} ,\\
u & = 0 & \quad & \text{ on } \partial \Omega_{\e, \eta},
\endaligned
\right.
\end{equation}
in a domain $\Omega_{\e, \eta}$ perforated with small holes in $\R^d$,
where $\e\in (0, 1]$ represents the scale of the minimal distances  between holes and $\eta \in (0, 1]$ the ratio  between  the scale of  sizes of holes and $\e$.
Let $F\in L^2(\Omega_{\e, \eta})$ and $f\in L^2(\Omega_{\e, \eta}; \R^d)$.
Under some general conditions on $\Omega_{\e, \eta}$, the Dirichlet problem \eqref{D-01} possesses a unique solution $u$
in $ W^{1, 2}_0(\Omega_{\e, \eta} )$.
Moreover, the solution satisfies the energy  estimate,
\begin{equation}\label{2-estimate}
\e ^{-1} \eta^{\frac{d-2}{2}}
 \| u\|_{L^2(\Omega_{\e, \eta})}
+ \| \nabla u \|_{L^2(\Omega_{\e, \eta})}
\le C \left\{  \| f\|_{L^2(\Omega_{\e, \eta})} + \e \eta^{\frac{2-d}{2}}  \| F \|_{L^2(\Omega_{\e, \eta})}\right\},
\end{equation}
for $d\ge 3$, where $C$ is independent of $\e$ and $\eta$.
The purpose of this paper is to investigate  the analogous estimates in the $L^p$ setting for $1< p< \infty$.
More precisely, we study the $W^{1, p}$ estimates,
\begin{equation}\label{W-1-p}
\left\{
\aligned
\| \nabla u \|_{L^p(\Omega_{\e, \eta} )}  & \le A_p(\e, \eta) \| f\|_{L^p(\Omega_{\e, \eta} )}
+B_p (\e, \eta) \|F\|_{L^p(\Omega_{\e, \eta})},\\
\|  u \|_{L^p(\Omega_{\e, \eta} )}  & \le C_p(\e, \eta) \| f\|_{L^p(\Omega_{\e, \eta} )}
+C_p (\e, \eta) \|F\|_{L^p(\Omega_{\e, \eta})},
\endaligned
\right.
\end{equation}
and are interested in the explicit dependance of  bounding constants
 $A_p(\e, \eta)$, $B_p(\e, \eta)$, $C_p(\e, \eta)$ and $D_p(\e, \eta)$  on the small parameters $\e$ and $\eta$.
 Note that by duality, $B_p(\e, \eta)=C_{p^\prime}(\e, \eta)$, where $p^\prime=\frac{p}{p-1}$.
 
Our study is motivated by the homogenization theory of boundary value problems for elliptic equations and systems 
in perforated domains, which are used to model
various processes in porous media and perforated materials.
The asymptotic behavior of the solutions $u=u_\e$, as $\e\to 0$,  has been studied extensively since 1970's.
In particular, if $\Omega_{\e, \eta}$ is a bounded Lipschitz domain perforated periodically  with $\eta=1$, given by \eqref{Omega},
and $u_\e$ is the solution of the Dirichlet problem:  
\begin{equation}\label{e-DP}
-\Delta u_\e =F \quad \text{ in } \Omega_{\e, \eta} \quad \text{  and } \quad
u_\e =0 \quad  \text{ on } \partial \Omega_{\e, \eta},
\end{equation}
then $u_\e/\e^2 \to \gamma_0 F$ weakly in $L^2(\Omega)$ for some constant $\gamma_0>0$.
If $\eta=\eta (\e)  \to 0$ as $\e \to 0$,
the asymptotic behavior of $u_\e$ is divided into three   cases.
In the case of large holes,  where $\sigma_\e=\e \eta^{\frac{2-d}{2}} \to 0$,
$u_\e/\sigma_\e^2 $ converges strongly in $L^2(\Omega)$ to $\gamma_1  F$ for some constant $\gamma_1>0$.
In the case of small holes,  where $\sigma_\e \to \infty$,
$u_\e$ converges strongly in $W^{1, 2}_0(\Omega)$ to $u_0$, where $u_0$ is the solution of the Dirichlet problem:
\begin{equation}\label{00-DP}
-\Delta u_0=F \quad \text{ in } \Omega \quad \text{  and } \quad u_0=0 \text{  on } \partial \Omega.
\end{equation}
In the critical case, where $\sigma_\e \to 1$,
the solution $u_\e$ converges weakly in $W^{1, 2}_0 (\Omega)$ to $u_0$,
where $u_0$ is the solution of the Dirichlet problem:
\begin{equation}\label{01-DP}
-\Delta u_0+\mu_* u_0=F \quad \text{ in } \Omega \quad \text{  and } \quad u_0=0 \text{  on } \partial \Omega, 
\end{equation}
and $\mu_*$ is a positive constant \cite{Marchenko-74, Lions-80, Murat-82a, Murat-82b, Murat-89, Hoang-2000}.
Similar results hold for the Stokes equations,
$-\Delta u_\e +\nabla p_\e=F$ and $\text{\rm div}(u_\e)=0$ in $\Omega_{\e, \eta}$, 
 with Dirichlet  condition $u_\e=0$ on $\partial \Omega_{\e, \eta}$ \cite{Allaire-90}.
We refer the reader to \cite{Allaire-90, JKO,  Allaire-1997, Jing-2020, Lu-2020, Blanc-22} and references therein for an introduction to the 
homogenization theory  of
Dirichlet  problems in perforated domains.

Our main   interest in this paper is in  uniform regularity  estimates for the Dirichlet problem
\eqref{D-01}. In the case $\eta=1$,
the $W^{1, p}$ estimates were established by N. Masmoudi \cite{ Masmoudi-2004} for Laplace's equation in an unbounded
 perforated domain given by \eqref{omega}.
Using a compactness method,
the results were extended recently in \cite{Shen-D} by the present author to the Stokes equations
 with $\eta=1$, whose asymptotic 
behavior is govern by the so-called Darcy's law.
This paper treats Laplace's equation in the vanishing volume case where $\eta \to 0$.
We are  able to establish $W^{1, p}$ estimates with optimal or  near optimal bounding constants  in a  general  non-periodic setting.
To the best of our knowledge,  this paper contains the first results on regularity estimates in $L^p$,  
which are uniform in both $\e\in  (0, 1]$ and $\eta\in (0, 1]$,
for Dirichlet problems
in perforated domains.
Theorems \ref{main-thm-2} and \ref{main-thm-3} will be useful in the study of optimal convergence rates for solutions  $u_\e$ of \eqref{e-DP}
 in $L^p$ and $W^{1, p}$ spaces
for $1<p< \infty$,
a topic we plan to return in a future study.

To state our main results,  let $Y=[-1/2,1/2]^d $ be a closed unit cube in $\R^d$ and
$\{T_k: k\in \mathbb{Z}^d \}$  a family of closed  subsets of $Y$.
Throughout the paper we shall assume that each $T_k$ is the closure of a bounded Lipschitz domain, 
$Y\setminus  {T_k}$ is connected, and that
\begin{equation}\label{condition-0}
B(0, c_0) \subset T_k,  \qquad \text{\rm dist} (\partial T_k, \partial Y)\ge c_0>0
\end{equation}
for some $c_0>0$.
Define
\begin{equation}\label{omega}
\omega_{\e, \eta} = \R^d \setminus \bigcup_{k\in \mathbb{Z}^d}  \e ( k + \eta {T_k}),
\end{equation}
where $0<\e,  \eta\le 1$.
Roughly speaking, the unbounded perforated domain $\omega_{\e, \eta}$ is obtained from 
$\R^d$ by removing  a hole $\e (k+\eta {T_k})$ of size  $\e \eta$ from each cube $\e (k+Y)$ of size $\e$.
The distances between holes  are bounded below by $c_0\e$.

\begin{thm}\label{main-thm-2}
Let $d\ge 3$ and $1< p< \infty$.
Let $\omega_{\e, \eta}$ be given by \eqref{omega}, where
 $\{T_k\}$ are the  closures of bounded domains with  uniform $C^1$ boundaries. 
Then, for any  $F \in L^p(\omega_{\e, \eta})$ and $f\in L^p(\omega_{\e, \eta}; \R^d)$, the Dirichlet problem,
\begin{equation}\label{D-omega}
-\Delta u = F+ \text{\rm div} (f) \quad \text{ in } \omega_{\e, \eta}
\quad \text{ and } \quad u=0 \quad \text{ on } \partial \omega_{\e, \eta},
\end{equation}
  has a unique solution in 
$W_0^{1, p}(\omega_{\e, \eta})$. Moreover, the solution satisfies the estimates,
\begin{equation}\label{m-e-1}
\| \nabla u \|_{L^p(\omega_{\e, \eta} )} \le
\left\{
\aligned
& C \, \e \eta^{1-\frac{d}{2}} \| F \|_{L^p(\omega_{\e, \eta})}
+ C_{ \delta}\,  \eta^{-d |\frac12 -\frac{1}{p}| -\delta} \| f \|_{L^p(\omega_{\e, \eta})}
& \quad & \text{ for } 1< p\le 2,\\
& C_{ \delta}\,  \e \eta^{1-d+\frac{d}{p}-\delta} \| F \|_{L^p(\omega_{\e, \eta})} 
+  C_{ \delta}\,  \eta^{-d |\frac12 -\frac{1}{p}| -\delta} \| f \|_{L^p(\omega_{\e, \eta})} & \quad & \text{ for } 2<p< \infty,
\endaligned
\right.
\end{equation}
and
\begin{equation}\label{m-e-2}
\| u\|_{L^p(\omega_{\e, \eta})}
\le \left\{
\aligned
& C\,  \e^2 \eta^{2-d} \| F \|_{L^p(\omega_{\e, \eta})}
+ C_{\delta} \, \e \eta^{1-\frac{d}{p}-\delta} \| f\|_{L^p(\omega_{\e, \eta})}
& \quad & \text{ for } 1< p< 2,\\
& C\,  \e^2 \eta^{2-d} \| F \|_{L^p(\omega_{\e, \eta})}
+ C\,  \e \eta^{1-\frac{d}{2} } \| f\|_{L^p(\omega_{\e, \eta})} & \quad & \text{ for } 2 \le  p< \infty,
\endaligned
\right.
\end{equation}
for any $\delta\in (0, 1)$,
where $C$ depends on $d$, $p$ and $\{T_k \}$, and $C_{ \delta}$ also depends on $\delta$.
\end{thm}

We point out that the powers of $\e$ in \eqref{m-e-1} and \eqref{m-e-2} are dictated by scaling.
In fact,  by rescaling,  it suffices to prove  the estimates for $\e=1$. The powers of $\eta$ in \eqref{m-e-1} and \eqref{m-e-2} are  either  optimal or
near optimal in the sense that
the estimates fail for any $\delta<0$.  
Indeed, by using a $Y$-periodic function $\chi_\eta$, which satisfies the equation 
\begin{equation}
-\Delta \chi_\eta =\eta^{d-2} \quad \text{ in } \omega_{1, \eta} \quad \text{ and } \quad
\chi_\eta =0 \quad \text{ on } \partial \omega_{1, \eta},
\end{equation}
we show that 
 lower bounds for $A_p (\e, \eta)$, $B_p (\e,\eta)$, $C_p(\e, \eta)$ and $D_p(\e, \eta)$ in a periodically 
 perforated domain $\omega_{\e, \eta}$, 
 are given by the corresponding bounding constants in \eqref{m-e-1} and \eqref{m-e-2} with $\delta=0$.
 More precisely, if the estimates in \eqref{W-1-p} hold for some $1< p< \infty$ with the best possible constants $A_p(\e, \eta)$, $B_p (\e, \eta)$,
 $C_p(\e, \eta)$ and $D_p(\e, \eta)$, then
 \begin{equation}
 A_p (\e, \eta) \ge c \, \eta^{-d |\frac12- \frac{1}{p}| },
 \end{equation}
 \begin{equation}
 D_p (\e, \eta)  \ge c\,  \e^2 \eta^{2-d},
 \end{equation}
and
 \begin{equation}
 B_p (\e, \eta) =C_{p^\prime} (\e, \eta) \ge \left\{
 \aligned
 &   c\, \e \eta^{1-\frac{d}{2}} & \quad & \text{ for } 1< p\le 2,\\
 & c\,  \e \eta^{1-d+\frac{d}{p}} & \quad & \text{ for } 2< p < \infty,
 \endaligned
 \right.
 \end{equation}
 where $c>0$ depends only $d$, $p$ and $T_k=T$.
  See Sections \ref{section-U} and \ref{section-GU} for details 
 as well as for $W^{1, p}$ estimates in the case $d=2$.
 In a forthcoming paper \cite{Shen-W}, 
 J. Wallace and the present author are able to establish the optimal estimates \eqref{m-e-1}-\eqref{m-e-2} with $\delta=0$ for $d\ge 3$
 in a periodically perforated domain. 
 We also obtain the optimal $W^{1, p}$ estimates in the case $d=2$.
The proof relies on a large-scale Lipschitz estimate in the periodic setting.
It would  be very interesting to extend the results in this paper  and \cite{Shen-W} to the Stokes equations.

Let $\Omega$ be a bounded domain in $\R^d$ and
\begin{equation}\label{Omega}
\Omega_{\e, \eta}
=\Omega \setminus 
\bigcup_k \e ( k +\eta T_k),
\end{equation}
where the union is taken over those $k$'s in $\mathbb{Z}^d$  for which $\e (k +Y) \subset \Omega$.

\begin{thm}\label{main-thm-3}
Let $d\ge 3$ and  $1< p< \infty$.
Let $\Omega_{\e, \eta}$ be given by \eqref{Omega}, where
$\Omega$ is a bounded $C^1$ domain and 
$\{ T_k: k \in \mathbb{Z}^d \}$  satisfies the same conditions as in Theorem \ref{main-thm-2}.
Then, for any $f\in L^p(\Omega_{\e, \eta}; \R^d)$ and $F\in L^p(\Omega_{\e, \eta})$, the Dirichlet problem \eqref{D-01}
 has a unique solution in 
$W_0^{1, p}(\Omega_{\e, \eta})$. Moreover, the solution satisfies the estimates in  \eqref{W-1-p}
with 
\begin{equation}\label{Ape}
A_p (\e, \eta) \le C_\delta \big\{ \min ( (\e\eta)^{-2}, \eta^{-d})  \big\}^{|\frac{1}{2}-\frac{1}{p}|+\delta},
\end{equation}
\begin{equation} \label{Bpe}
B_p(\e, \eta) \le
\left\{
\aligned
& C \e\eta  \big\{ \min ((\e\eta)^{-2} ,  \eta^{-d} ) \big\}^{\frac12}  & \quad & \text{ if } 1< p\le 2 ,\\
&C_\delta\, \e \eta  \big\{ \min ((\e\eta)^{-2}, \eta^{-d}) \big\}^{1-\frac{1}{p} +\delta } & \quad & \text{ if } 2< p< \infty,
\endaligned
\right.
\end{equation}
\begin{equation} \label{Cpe}
C_p(\e, \eta) \le
\left\{
\aligned
& C \e\eta  \big\{ \min ((\e\eta)^{-2} ,  \eta^{-d} ) \big\}^{\frac12}  & \quad & \text{ if } 2\le  p< \infty  ,\\
&C_\delta\, \e \eta  \big\{ \min ((\e\eta)^{-2}, \eta^{-d}) \big\}^{\frac{1}{p} +\delta } & \quad & \text{ if } 1< p< 2,
\endaligned
\right.
\end{equation}
and 
\begin{equation}\label{Dpe}
D_p (\e, \eta)
\le C (\e \eta)^2 \min ( (\e\eta)^{-2}, \eta^{-d}),
\end{equation}
for any $\delta\in (0, 1)$, where $C$ depends on $d$, $p$, $\{T_k\}$ and $\Omega$, and $C_\delta$ also depends on $\delta$.
\end{thm}

Theorem \ref{main-thm-3} shows that the estimates in \eqref{m-e-1} and \eqref{m-e-2} continue to hold with $\Omega_{\e, \eta}$ in the place of $\omega_{\e, \eta}$, 
and that
 if $\sigma_\e = \e \eta^{1-\frac{d}{2}}\gg 1$, better estimates hold in  the bounded perforated domain.
 We mention that analogous estimates to those in Theorem \ref{main-thm-3}  are also obtained for the case $d=2$. 
 See Section \ref{section-B} for details.

We now describe our approach to Theorems \ref{main-thm-2} and  \ref{main-thm-3}.
The first step is to establish the estimate \eqref{m-e-2}  in the case $2\le p< \infty$,
by using the test functions,
$$
v_\ell =\min \{ |u|^{p-2}, \ell^{ p-2} \},
$$
where $\ell \ge 1$,
in the weak formulation of the elliptic equation $-\Delta u =F +\text{\rm div}(f)$ in $\omega_{\e, \eta}$.
This is a classical method  that goes back to  J. Moser \cite{Moser-60} and was used to establish $L^\infty$ estimates for
weak solutions of elliptic equations in divergence form with bounded measurable coefficients.
The approach was also used in \cite{Masmoudi-2004} in the case $\eta=1$.
By using a  Poincar\'e inequality for functions in $W^{1, 2} (Y)$ that vanish on $B(0, \eta)$,
we are able to deduce the estimate \eqref{m-e-2} for $\| u \|_{L^p(\omega_{\e, \eta})}$  in the case  $2\le p< \infty$.
This estimate is sharp and requires no smoothness condition on $ \{T_k: k \in \mathbb{Z}^d\}$.
Next, we use a localization argument to show that
$$
\| \nabla u \|_{L^p(\omega_{\e, \eta})}
\le C \left\{ (\e \eta)^{-1} \| u\|_{L^p(\omega_{\e, \eta})}  + \| F \|_{L^p(\omega_{\e, \eta})} + \| f\|_{L^p(\omega_{\e, \eta})} \right\},
$$
for $2< p< \infty$, under the assumption that the boundaries of holes $\{T_k\}$ are uniformly $C^1$.
As a result, we obtain an estimate of $\|\nabla u\|_{L^{p_0} (\omega_{\e, \eta})} $ for  any large $p_0>2$.
Finally, we apply the Riesz-Thorin Interpolation Theorem for $2<p< p_0$, utilizing the $L^2$ energy  estimate in \eqref{2-estimate} and
the estimate of $\nabla u$ in $L^{p_0}$.
For any $\delta\in (0, 1)$, we obtain the estimate \eqref{m-e-1} for $2<p<\infty$ by choosing $p_0$ sufficiently large.
The case $1<p<2$ follows by a duality argument.

We remark that the approach outlined above works equally well for the bounded perforated domain $\Omega_{\e, \eta}$.
The additional application  of the Poincar\'e inequality on the bounded domain $\Omega$ contributes to  the appearance of the factor 
$\min ((\e \eta)^{-2}, \eta^{-d})$ in Theorem \ref{main-thm-3}, which takes value  $(\e \eta)^{-2}$ and yields better estimates
 in the case $ \sigma_\e \gg 1$ (the case of small holes).


\section{Local estimates}\label{section-L}

Recall that $Y=[-1/2, 1/2]^d$.
We begin with a Poincar\'e inequality.

\begin{lemma}\label{p-lemma}
Let $d\ge 2$ and $1< p< \infty$.
Suppose that  $u\in W^{1, p}(Y)$ and $u=0$ on $B(0, \eta)$ for some $0< \eta< 1/4$. Then
\begin{equation}\label{P}
\int_Y |u|^p\, dx
\le C \int_Y |\nabla u|^p\, dx
\cdot 
\left\{
\aligned
& \eta^{p-d} & \quad & \text{ if \ \ } 1< p< d,\\
& | \ln \eta|^{p-1} & \quad & \text{ if \ \ } p=d,\\
& 1 & \quad & \text{ if \ \ } d<p< \infty,
\endaligned
\right.
\end{equation}
where $C$ depends on $d$ and $p$.
\end{lemma}

\begin{proof}
This case $p=2$ is more or less well known. 
The proof for the general case is similar.
We provide a proof for the reader's convenience.
Using $u=0$ on $B(0, \eta)$, we may write 
$$
u(x)=u(r\omega) -u(\eta \omega)=\int_{\eta}^r \omega \cdot \nabla u(t \omega) dt
$$
for any $x\in Y$, where $r=|x|$ and $\omega=x/|x|$.
It follows by H\"older's inequality  that
$$
|u(x)|^p\le \int_\eta^r |\nabla u(t\omega)|^p t^{d-1}\, dt 
\left(\int_\eta^r t^{-\frac{d-1}{p-1}} dt\right)^{p-1}
$$
for $1< p< \infty$.
Thus,
$$
\int_Y | u|^p\, dx
\le C \int_Y |\nabla u|^p\, dx
\left(\int_\eta^d t^{-\frac{d-1}{p-1}} dt\right)^{p-1},
$$
from which the inequalities in \eqref{P} follow readily.
\end{proof}

For $1<p<\infty$,
let $W_0^{1, p}(\Omega)$ denote the completion of $C_0^\infty(\Omega)$ with respect to the norm $\|\nabla \psi\|_{L^p(\Omega)}$ and
$W^{-1, p^\prime}(\Omega)$ the dual of $W_0^{1, p} (\Omega)$.

The following theorem was proved by D. Jerison and C. Kenig in \cite{JK}.

\begin{thm}\label{JK-thm}
Let $\Omega$ be a bounded Lipschitz domain in $\R^d$, $d\ge 2$.
Let $F \in W^{-1, p}(\Omega)$. 
Then  the Dirichlet problem, 
\begin{equation}\label{D-2}
\left\{
\aligned
-\Delta u & =F & \quad & \text{ in } \Omega,\\
u& =0 & \quad  & \text{ on } \partial \Omega,
\endaligned
\right.
\end{equation}
has a unique solution in $W_0^{1, p} (\Omega)$, if 
\begin{equation}\label{Lip-R}
 \left| \frac{1}{p}-\frac12 \right| <
 \left\{
 \aligned
 & \frac16 +\delta & \quad & \text{ for } d\ge 3,\\
&  \frac14 +\delta & \quad & \text{ for } d=2,
\endaligned
\right.
\end{equation}
for some $\delta>0$ depending on $\Omega$. Moreover, the solution satisfies 
\begin{equation}\label{JK-estimate}
\| \nabla u\|_{L^p(\Omega)}
\le C\,  \| F \|_{W^{-1, p} (\Omega)},
\end{equation}
where $C$ depends on $d$, $p$ and the Lipschitz character of $\Omega$.
If $\Omega$ is  a bounded $C^1$ domain,
the estimate \eqref{JK-estimate} holds for $1< p< \infty$.
\end{thm}

\begin{remark}\label{re-char}
{\rm
Let $\Omega$ be a bounded Lipschitz domain.
We may cover $\partial\Omega$ by a finite number of balls $\{ B(x_i, r_0): i=1, 2, \dots, N\}$, where $x_i\in \partial\Omega$ and
$r_0>0$,  so that after a translation and rotation of the coordinate system,
$$
\aligned
B(x_i, 10 r_0)\cap \Omega  & = B(x_i, 10 r_0) \cap \left\{ (x^\prime, x_d)\in \R^d: \ x_d > \psi_i (x^\prime) \right\},\\
B(x_i, 10 r_0)\cap \partial \Omega  & = B(x_i, 10 r_0) \cap \left\{ (x^\prime, x_d)\in \R^d: \ x_d = \psi_i (x^\prime) \right\},
\endaligned
$$
where $\psi_i$ is a Lipschitz function in $\R^{d-1}$ with $\|\nabla \psi_i \|_\infty \le M$.
By a constant $C$ depending on the Lipschitz character of $\Omega$, 
we mean that $C$ depends on $(N, M)$.
If $\{\Omega_k \}$ is a family of bounded Lipschitz domains with the same $N$ and $M$,
we shall call $\{ \Omega_k \}$ a family of bounded domains with uniform Lipschitz boundaries.
Thus the estimates \eqref{JK-estimate} for  the domain $\Omega_k$   hold  uniformly with a constant $C$ independent of $k$,
if $p$ satisfies \eqref{Lip-R} (with a uniform $\delta>0$).
In the case of $C^1$ domains, we also require that the modules of continuity for $\nabla \psi_i$
are  bounded uniformly by a continuous and increasing  function $\xi (t)$ on $[0, \infty)$ with $\xi (0)=0$.
We call such $\{\Omega_k \}$ a family of bounded domains with uniform $C^1$ boundaries.
In this case the estimates \eqref{JK-estimate} for $\Omega_k$
hold uniformly with a constant $C$ independent of $k$ for $1<p<\infty$.
}
\end{remark}

\begin{remark}\label{remark-2.1}
{\rm
Let $F\in L^q(\Omega)$ and $f\in L^p(\Omega; \R^d)$, where $d^\prime=\frac{d}{d-1} < p< \infty$ and $\frac{1}{q}=\frac{1}{p}+\frac{1}{d}$.
It follows from Theorem \ref{JK-thm} that 
 if $\Omega$ is a bounded $C^1$ domain in $\R^d$, $d\ge 2$, and $u$ is a weak solution of
\begin{equation}\label{D-4}
-\Delta u  =F+\text{\rm div}(f)   \quad  \text{ in } \Omega \quad \text{ and } \quad 
u  =0 \quad  \text{ on } \partial \Omega,
\end{equation}
then,
\begin{equation}\label{Jk-F}
\|\nabla u \|_{L^p(\Omega)} \le C \left\{ \| F\|_{L^q(\Omega)} + \| f\|_{L^p(\Omega)} \right\}.
\end{equation}
If $\Omega$ is a bounded Lipschitz domain, the estimate above holds under the additional conditions that
$ p\le 3$ for $d\ge 3$ and $ p\le 4$ for $d=2$.
}
\end{remark}

\begin{lemma}\label{local-lemma-1}
Let  $d\ge 2$ and $d^\prime < p< \infty$.
Let $T$ be the closure of an open subset of $Y$ with $C^1$ boundary.
Assume that $ 0\in T$ and $Y\setminus  {T}$ is connected.
Let $u\in W^{1, p}(  \tY \setminus {\eta T})$ be a weak solution of
\begin{equation}\label{local-1-0}
\left\{
\aligned
-\Delta u& = F +\text{\rm div}(f) & \quad & \text{ in } \tY \setminus { \eta T},\\
u&=0 & \quad & \text{ on } \partial (\eta T),
\endaligned
\right.
\end{equation}
where $\tY =(1+c_0)Y$ and $0< \eta\le 1$.
Then
\begin{equation}\label{local-1-2}
\int_{Y\setminus \eta T} |\nabla u|^p\, dx
\le C\left\{
 \eta^{-p} 
\int_{\tY \setminus \eta T} |u|^p\, dx
+ \int_{\tY \setminus \eta T} |F|^p\, dx
+ \int_{\tY\setminus \eta T} |f|^p\, dx \right\},
\end{equation}
where $C$ depends on $d$, $p$, $c_0$ and $T$.
If $\partial T$ is Lipschitz, the estimate above holds under the additional conditions on $p$ in  \eqref{Lip-R}.
\end{lemma}

\begin{proof}
We may assume $\eta \in (0, 1/4)$. The case $\eta\in [1/4, 1]$ follows from the $W^{1, p}$ estimate \eqref{Jk-F}
for Laplace's equation  in $C^1$ and Lipschitz domains, by a localization argument.
In fact, the case $\eta\in (0, 1/4)$ also follows readily  from  \eqref{Jk-F} by a localization argument.
To see this, choose a cut-off function $\varphi\in C_0^\infty (\tY\setminus \eta T)$ such that $0\le \varphi\le 1$ in $\tY \setminus \eta T$, 
$\varphi =1$ in $Y\setminus 2\eta T$, and
$$
\aligned
& |\nabla \varphi |\le C \eta^{-1}, \ \ \ |\nabla^2 \varphi| \le C \eta^{-2} \quad \text{ in } 2\eta T\setminus \eta T,\\
& |\nabla \varphi | + |\nabla^2 \varphi| \le C \quad \text{ in } \tY \setminus Y.
\endaligned
$$
Using
\begin{equation}\label{local-f}
-\Delta (u \varphi)
=F \varphi +\text{\rm div} (f \varphi) - f \cdot \nabla \varphi - 2\,  \text{\rm div} ( u \nabla \varphi) + u \Delta \varphi
\end{equation}
in $\R^d$, we obtain 
$$
\| \nabla (u \varphi)\|_{L^p(\R^d)}
\le C \left\{  \| f  \varphi \|_{L^p(\R^d)} + \| u \nabla \varphi \|_{L^p(\R^d)}
+ \| F \varphi\|_{L^q(\R^d)} + \| f \nabla \varphi \|_{L^q (\R^d)}
+ \| u \Delta \varphi \|_{L^q (\R^d)}\right\},
$$
where $d^\prime< p< \infty$ and  $\frac{1}{q} =\frac{1}{p} +\frac{1}{d}$.
This gives
\begin{equation}\label{local-1-3}
\| \nabla u \|_{L^p(Y\setminus  2\eta T)}
\le C \left\{ \| f\|_{L^p(\tY \setminus \eta T)} + \| F \|_{L^p(\tY \setminus \eta T)}
+ \eta^{-1} \| u \|_{L^p(\tY \setminus \eta T)} \right\},
\end{equation}
where we have used the observation  $\| u\|_{L^q(2\eta T \setminus \eta T)} \le C \eta \| u\|_{L^p(2\eta T\setminus \eta T)}$.

Finally, if  $\partial T$ is $C^1$ and $u=0$ on $\partial T$, it follows from \eqref{Jk-F} by a rescaling argument that
\begin{equation}\label{local-1-4}
\| \nabla u \|_{L^p(2\eta T\setminus \eta T)}
\le C \left\{ \eta^{-1} \| u\|_{L^p (3\eta T\setminus \eta T)}
+ \| f \|_{L^p(3\eta T\setminus \eta T)}
+ \eta \| F \|_{L^p(3\eta T\setminus  \eta T)}
\right\}
\end{equation}
for $1<p< \infty$.
If $\partial T$ is Lipschitz, the same is true under the additional conditions in \eqref{Lip-R}.
The estimate \eqref{local-1-4}, together with \eqref{local-1-3},  yields \eqref{local-1-2}.
\end{proof}


\section{Global estimates}\label{section-G}

Let $\{T_k: k \in \mathbb{Z}^d\}$ be a family of closed subsets of $Y=[-1/2, 1/2]^d$.
Assume that each  $T_k$ is the  closure of a bounded Lipschitz domain and satisfies  the condition 
\eqref{condition-0}. 
We also assume that $Y\setminus {T_k}$ is connected.

\begin{lemma}\label{lemma-3.1}
Let $\omega_{\e, \eta}$ be given by \eqref{omega}.
Then, for any $u\in W_0^{1, 2}(\omega_{\e, \eta})$, 
\begin{equation}\label{3.1-0}
\int_{\omega_{\e, \eta}} |u|^2\, dx
\le C\int_{\omega_{\e, \eta} } |\nabla u|^2\, dx
\cdot 
\left\{
\aligned
 & \e^2 \eta^{2-d} & \quad & \text{  if   \ \ } d\ge 3,\\
 &\e^2  |\ln (\eta/2)| &\quad &  \text{ if \ \ } d=2,
 \endaligned
 \right.
 \end{equation}
 where $0< \e, \eta\le 1$ and   $C$ depends on $d$ and $c_0$ in \eqref{condition-0}.
 \end{lemma}

\begin{proof}
By rescaling we may assume $\e=1$.
Assume $d\ge 3$. By applying Poincar\'e's inequality \eqref{P}, we obtain 
$$
\int_{k + (Y\setminus \eta T_k)} |u|^2 \, dx
\le C \eta^{2-d} \int_{k +(Y\setminus \eta T_k)} |\nabla u|^2\, dx
$$ 
for each $k \in \mathbb{Z}^d$, where we have used the assumption  $B(0, c_0)\subset T_k$.
The inequality \eqref{3.1-0} with $\e=1$ and $d\ge 3$ follows by summing the inequalities above over $k \in \mathbb{Z}^d$.
The proof for the case $d=2$ is the same.
\end{proof}

\begin{remark}\label{remark-3.1}
{\rm
Let $\Omega_{\e, \eta}$ be given by \eqref{Omega}, where $\Omega$ is a bounded Lipschitz domain in $\mathbb{R}^d$, $d\ge 2$.
Then,   for $u\in W_0^{1, 2} (\Omega_{\e, \eta})$,
\begin{equation}\label{u-P}
\int_{\Omega_{\e, \eta}} |u|^2\, dx
\le C \int_{\Omega_{\e, \eta}} |\nabla u|^2\, dx,
\end{equation}
where $C$ depends on $\Omega$.
This follows by extending $u$ by zero to $\Omega$ and applying Poincar\'e's inequality for $\Omega$.
Moreover, 
\begin{equation}\label{3.1-O}
\int_{\Omega_{\e, \eta}} |u|^2\, dx
\le C\int_{\Omega_{\e, \eta} } |\nabla u|^2\, dx
\cdot 
\left\{
\aligned
 & \e^2 \eta^{2-d} & \quad & \text{  if   \ \ } d\ge 3,\\
 &\e^2  |\ln (\eta/2)| &\quad &  \text{ if \ \ } d=2,
 \endaligned
 \right.
 \end{equation}
for any $u\in W_0^{1, 2} (\Omega_{\e, \eta})$.
To see \eqref{3.1-O}, we may assume $\e$ is sufficiently small.
For otherwise, the desired estimate follows from \eqref{u-P}.
We cover $\Omega$ by a family of cubes $\{ \e (k + Y) \}$ of size $\e$, where $k\in \mathbb{Z}^d$ and
$\e (k +Y) \cap \Omega \neq \emptyset$.
The case $\e(k+Y)\subset \Omega$ is handled in the same manner as in the proof of Lemma \ref{lemma-3.1}.
If $\e(k_0+Y) \cap \partial \Omega \neq \emptyset$ for some $k_0\in \mathbb{Z}^d$, we may find $k_1\in \mathbb{Z}^d$ such that
$\e (k_1 + Y)\subset \Omega$ and 
$$
\e (k_0 +Y)\subset 5 \e (k_1+Y).
$$
It follows from the proof of Lemma \ref{p-lemma} that 
$$
\int_{\e (k_0 +Y)} |u|^2\, dx
\le \int_{5 \e (k_1 +Y)} |u|^2\, dx
\le C \e^2 \eta^{2-d} \int_{5 \e (k_1+Y)} |\nabla u|^2\, dx
$$
for $d\ge 3$,
where we have extended $u$ by zero to $\R^d$.
As a result,
$$
\int_{\Omega_{\e, \eta}}  |u|^2\, dx
\le C \e^2 \eta^{2-d} \int_{\R^d} |\nabla u|^2\, dx 
\le C \e^2 \eta^{2-d} \int_{\Omega_{\e, \eta}} |\nabla u|^2\, dx.
$$
The proof for the case $d=2$ is the same.
}
\end{remark}

It is not hard  to see that 
$$
W_0^{1, 2}(\omega_{\e, \eta}) = \left\{ u\in W^{1, 2} (\R^d): u=0 \text{ on } \R^d\setminus  \omega_{\e, \eta} \right\}.
$$
Using Lemma \ref{lemma-3.1} and the Lax-Milgram Theorem, one may show that, given  $F \in L^2(\omega_{\e, \eta})$ and
$f\in L^2(\omega_{\e, \eta }; \R^d)$, the Dirichlet problem 
\begin{equation}\label{D-3-0}
\left\{
\aligned
-\Delta u & = F +\text{\rm div} (f)  & \quad & \text{ in } \omega_{\e, \eta}, \\
u & =0 & \quad & \text{ on } \partial \omega_{\e, \eta},
\endaligned
\right.
\end{equation}
has a unique solution in $W_0^{1, 2} (\omega_{\e, \eta})$.
Moreover, if $d\ge 3$,
 the solution satisfies
\begin{equation}\label{3.1-e}
\| \nabla u\|_{L^2(\omega_{\e, \eta})}
+ \e^{-1} \eta^{\frac{d-2}{2}} \|u\|_{L^2(\omega_{\e, \eta})}
\le C \left\{
 \e \eta^{\frac{2-d}{2}} \| F \|_{L^2(\omega_{\e, \eta})}
+ \|f \|_{L^2(\omega_{\e, \eta})} \right\},
\end{equation}
and  if  $d=2$, 
\begin{equation}\label{3.1-2e}
\aligned
\| \nabla u\|_{L^2(\omega_{\e, \eta})}
& + \e^{-1}  | \ln (\eta/2)|^{-1/2}  \|u\|_{L^2(\omega_{\e, \eta})}\\
& \le C \left\{
 \e  |\ln (\eta/2)|^{1/2}   \| F \|_{L^2(\omega_{\e, \eta})}
+ \|f \|_{L^2(\omega_{\e, \eta})} \right\}.
\endaligned
\end{equation}
The constants  $C$ in \eqref{3.1-e} and \eqref{3.1-2e}
depend only on $d$ and $c_0$ in \eqref{condition-0}.

The following is one of the key estimates in this paper.

\begin{thm}\label{thm-3.0}
Let $\omega_{\e, \eta}$ be given by \eqref{omega} and $2\le p< \infty$.
Given $F\in L^2(\omega_{\e, \eta}) \cap L^p(\omega_{\e, \eta})$ and
$f\in L^2(\omega_{\e, \eta};  \R^d) \cap L^p (\omega_{\e, \eta};  \R^d)$,
let $u$ denote  the weak solution of \eqref{D-3-0} in $W_0^{1, 2} (\omega_{\e, \eta})$.
Then $u\in L^p (\omega_{\e, \eta})$ and
\begin{equation}\label{g-3-0}
\| u \|_{L^p(\omega_{\e, \eta})}
\le C \left\{
\e^2 \eta^{2-d} \| F\|_{L^p(\omega_{\e, \eta})}
+ \e \eta^{\frac{2-d}{2}} \| f\|_{L^p(\omega_{\e, \eta})} \right\},
\end{equation}
where $d\ge 3$ and  $C$ depends on $d$, $p$ and $c_0$.
If $d=2$, we have
\begin{equation}\label{g-3-1}
\| u \|_{L^p(\omega_{\e, \eta})}
\le C \left\{
\e^2 |\ln (\eta/2)|  \| F\|_{L^p(\omega_{\e, \eta})}
+ \e |\ln (\eta/2)|^{1/2}  \| f\|_{L^p(\omega_{\e, \eta})} \right\}.
\end{equation}
\end{thm}

\begin{proof}
We give the proof for the case $d\ge 3$.
The proof for the case $d=2$ is similar.

For $\ell \ge 1$, let 
$$
v_\ell =  \min \{ |u|^{p-2}, \ell^{p-2} \} u , 
$$
where $p>2$ (the case $p=2$ is given by \eqref{3.1-e} and \eqref{3.1-2e}).
Note that
$$
\aligned
\nabla v_\ell  & = (p-2) |u|^{p-2} \chi_{\{|u|< \ell\} }  \nabla u
+ \min (|u|^{p-2}, \ell^{p-2}) \nabla u\\
&= (p-1) |u|^{p-2} \chi_{ \{ |u|< \ell\}} \nabla u + \ell^{p-2} \chi_{ \{ |u|\ge \ell \} } \nabla u.
\endaligned
$$
It follows that $v_\ell \in W^{1, 2}_0(\omega_{\e, \eta})$ and 
$$
\int_{\omega_{\e, \eta}} \nabla u \cdot \nabla v_\ell \, dx
=\int_{\omega_{\e, \eta}} F v_\ell\, dx
-\int_{\omega_{\e, \eta}} f \cdot \nabla v_\ell\, dx.
$$
This gives
$$
\aligned
& c \int_{\omega_{\e, \eta}} \min (|u|^{p-2}, \ell^{p-2}) |\nabla u|^2\, dx\\
&\le \int_{\omega_{\e, \eta}} |F| \min (|u|^{p-2}, \ell^{p-2}) |u|\, dx
+ \int_{\omega_{\e, \eta}} |f| \min (|u|^{p-2}, \ell^{p-2}) |\nabla u|\, dx,
\endaligned
$$
where $c$ depends on $p$.
By using the Cauchy inequality we obtain 
\begin{equation}\label{g-3-1-1}
\aligned
&  \int_{\omega_{\e, \eta}} \min (|u|^{p-2}, \ell^{p-2}) |\nabla u|^2\, dx\\
&\le  C \left\{ \int_{\omega_{\e, \eta}} |F| \min (|u|^{p-2}, \ell^{p-2}) |u|\, dx
+ \int_{\omega_{\e, \eta}} |f|^2 \min (|u|^{p-2}, \ell^{p-2}) \, dx\right\}.
\endaligned
\end{equation}

Next, let
$$
w_\ell = \min (|u|^{\frac{p}{2}-1}, \ell^{\frac{p}{2}-1}) u.
$$
Note that $w\in W^{1, 2} _0(\omega_{\e, \eta})$ and
$$
|\nabla w_\ell| \approx
\min (|u|^{\frac{p}{2}-1}, \ell^{\frac{p}{2}-1}) |\nabla u|.
$$
It follows by Lemma \ref{lemma-3.1} and \eqref{g-3-1-1}  that
$$
\aligned
\int_{\omega_{\e, \eta}}
|w_\ell|^2\, dx
 & \le C \e^2 \eta^{2-d} \int_{\omega_{\e, \eta}} |\nabla w_\ell|^2\, dx\\
 &\le C \e^2 \eta^{2-d}
 \Bigg\{
 \| F \|_{L^p(\omega_{\e, \eta})}
 \left(\int_{\omega_{\e, \eta}}  \left( \min( |u|^{p-2}, \ell^{p-2}) |u|\right)^{p^\prime}\, dx \right)^{1/p^\prime}\\
& \qquad\qquad\qquad\qquad
 + \| f\|_{L^p(\omega_{\e, \eta})}^2
 \left(\int_{\omega_{\e, \eta}} \left(\min (|u|^{p-2}, \ell^{p-2})\right)^{\frac{p}{p-2}} \, dx \right)^{1-\frac{2}{p}}
 \Bigg\},
\endaligned
$$
where we have used H\"older's inequality for the last step.

Finally, we observe that 
$$
|w_\ell|^2 = \min (|u|^{p-2}, \ell^{p-2} ) |u|^2,
$$
and  thus 
$$
\aligned
 \left( \min( |u|^{p-2}, \ell^{p-2}) |u|\right)^{p^\prime} & \le |w_\ell|^2,\\
  \left(\min (|u|^{p-2}, \ell^{p-2})\right)^{\frac{p}{p-2}} 
  &\le |w_\ell|^2,
  \endaligned
 $$
 where we have used the assumption  $p>2$.
 As a result, we obtain 
 $$
 \aligned
 \| w_\ell\|_{L^2(\omega_{\e, \eta})}^2
  & \le C \e^2 \eta^{2-d}
 \left\{
 \| F \|_{L^p(\omega_{\e, \eta})} \| w_\ell \|_{L^2(\omega_{\e, \eta})}^{2/p^\prime}
 + \| f\|_{L^p(\omega_{\e, \eta})}^2
 \| w_\ell \|_{L^2(\omega_{\e, \eta})}^{2-\frac{4}{p}}\right\}\\
 &\le  C (\e^2 \eta^{2-d})^p \| F\|_{L^p(\omega_{\e, \eta})}^p
 + C (\e^2 \eta^{2-d})^{\frac{p}{2} }\| f\|_{L^p(\omega_\e, \eta)}^p
 + (1/2) \| w_\ell \|_{L^2(\omega_{\e, \eta})}^2,
 \endaligned
$$
where we have used Young's inequality for the last step.
Thus, we have proved that
$$
\int_{\omega_{\e, \eta}}
|w_\ell|^2\, dx
\le C (\e^2 \eta^{2-d})^p
\int_{\omega_{\e, \eta}} |F|^p\, dx
+ C (\e^2 \eta^{2-d})^{p/2} \int_{\omega_{\e, \eta}} |f|^p\, dx,
$$
from which the estimate \eqref{g-3-0} follows by letting $\ell \to \infty$ and using  Fatou's Lemma.
\end{proof}

The estimates in Theorem \ref{thm-3.0} also hold for a bounded perforated domain.

\begin{thm}\label{thm-3.1}
Let $\Omega_{\e, \eta}$ be given by \eqref{Omega}, where $\Omega$ is a bounded Lipschitz domain.
Let $2\le p< \infty$.
Given $F\in L^p(\Omega_{\e, \eta})$ and $f\in L^p(\Omega_{\e, \eta}; \R^d)$,
let $u$ denote the weak solution of \eqref{D-01} in $W_0^{1, 2}(\Omega_{\e, \eta})$.
Then $u\in L^p (\Omega_{\e, \eta})$ and
\begin{equation}\label{g-3.1-0}
\| u \|_{L^p(\Omega_{\e, \eta})}
\le C \left\{
\min (1, \e^2\eta^{2-d}) \| F \|_{L^p(\Omega_{\e, \eta})}
+ \min (1, \e \eta^{\frac{2-d}{2}}) \| f\|_{L^p(\Omega_{\e, \eta})} \right\},
\end{equation}
where $d\ge 3$ and $C$ depends on $d$, $p$, $c_0$ and $\Omega$.
If $d=2$, we have 
\begin{equation}\label{g-3.1-1}
\| u \|_{L^p(\Omega_{\e, \eta})}
\le C \left\{
\min (1, \e^2 |\ln (\eta/2)|) \| F \|_{L^p(\Omega_{\e, \eta})}
+ \min (1, \e  |\ln (\eta/2)|^{1/2} ) \| f\|_{L^p(\Omega_{\e, \eta})} \right\}.
\end{equation}
\end{thm}

\begin{proof}
By Remark \eqref{remark-3.1}, 
 the estimates in \eqref{3.1-0} continue to hold if $\omega_{\e, \eta}$ is replaced  by $\Omega_{\e, \eta}$.
 Consequently, an inspection of the proof of Theorem \ref{thm-3.0}
 shows that the estimates \eqref{g-3-0} and \eqref{g-3-1} hold with $\Omega_{\e, \eta}$ in the place of $\omega_{\e, \eta}$.
Moreover, by using the  Poincar\'e inequality \eqref{u-P}, 
 the same argument as in the proof of Theorem \ref{thm-3.0} also gives
\begin{equation}\label{g-3.1-2}
\| u\|_{L^p(\Omega_{\e, \eta})}
\le C \left\{ \| F \|_{L^p(\Omega_{\e, \eta})}
+ \| f\|_{L^p(\Omega_{\e, \eta})} \right\},
\end{equation}
for $d\ge 2$ and $2\le p< \infty$, where $C$ depends on $d$, $p$ and $\Omega$.
By combining \eqref{g-3.1-2} with \eqref{g-3-0} and \eqref{g-3-1},
we obtain \eqref{g-3.1-0} and \eqref{g-3.1-1}.
\end{proof}



\section{Estimates of $u$}\label{section-U}

For $1< p<\infty$ and $\e, \eta\in (0, 1]$, we introduce operator norms,
\begin{equation}\label{A-p}
\aligned
A_p (\e, \eta)
=\sup \Big\{   \| \nabla u\|_{L^p(\omega_{\e, \eta})}:
   &\  u\in W^{1, 2}_0(\omega_{\e, \eta})  \text{ is a weak solution of } -\Delta u=\text{\rm div} (f)  \text{ in } \omega_{\e, \eta} \\
 & \text{ with }
f\in L^p(\omega_{\e, \eta}; \R^d) \cap L^2 (\omega_{\e, \eta}; \R^d) \text{ and } 
\| f\|_{L^p(\omega_{\e, \eta})} =1 \Big\},
\endaligned
\end{equation}
\begin{equation}\label{B-p}
\aligned
B_p (\e, \eta)
=\sup \Big\{   \| \nabla u\|_{L^p(\omega_{\e, \eta})}:
   &\  u\in W_0^{1, 2} (\omega_{\e, \eta})  \text{ is a weak solution of } -\Delta u=F \text{ in } \omega_{\e, \eta} \\
 & \text{ with }
F\in L^p(\omega_{\e, \eta}) \cap L^2 (\omega_{\e, \eta}) \text{ and } 
\| F\|_{L^p(\omega_{\e, \eta})} =1 \Big\},
\endaligned
\end{equation}
\begin{equation}\label{C-p}
\aligned
C_p (\e, \eta)
=\sup \Big\{   \| u\|_{L^p(\omega_{\e, \eta})}:
   & \ u\in W_0^{1, 2} (\omega_{\e, \eta})  \text{ is a weak solution of } -\Delta u=\text{\rm div} (f)  \text{ in } \omega_{\e, \eta} \\
 & \text{ with }
f\in L^p(\omega_{\e, \eta}; \R^d) \cap L^2 (\omega_{\e, \eta}; \R^d) \text{ and } 
\| f\|_{L^p(\omega_{\e, \eta})} =1 \Big\},
\endaligned
\end{equation}
\begin{equation}\label{D-p}
\aligned
D_p (\e, \eta)
=\sup \Big\{   \| u\|_{L^p(\omega_{\e, \eta})}:\ 
   & u\in W_0^{1, 2} (\omega_{\e, \eta})  \text{ is a weak solution of } -\Delta u=F \text{ in } \omega_{\e, \eta} \\
 & \text{ with }
F\in L^p(\omega_{\e, \eta}) \cap L^2 (\omega_{\e, \eta}) \text{ and } 
\| F\|_{L^p(\omega_{\e, \eta})} =1 \Big\}.
\endaligned
\end{equation}

\begin{lemma}\label{lemma-U-1}
Let $A_p(\e, \eta), B_p(\e, \eta)$, $C_p(\e, \eta)$ and $D_p(\e, \eta)$ be defined above.
Then
\begin{equation}
A_p (\e, \eta) = A_p(1, \eta), \quad \text{ \ } \quad
B_p (\e, \eta) = \e B_p (1, \eta),
\end{equation}
\begin{equation}
C_p (\e, \eta) =\e\,  C_p(1, \eta), \quad \text{ \ } \quad
D_p (\e, \eta) = \e^2 D_p (1, \eta).
\end{equation}
\end{lemma}

\begin{proof}
This follows by a simple rescaling, using the observation that 
if  $-\Delta u=F +\text{\rm div}(f)$ in $\omega_{\e, \eta}$ and
 $v(x)=u(\e x)$, then 
 $$
 -\Delta v (x) = \e^2 F(\e x) + \e \text{\rm div} \{  f(\e x )\}
 $$
for $x \in \omega_{1, \eta}$.
\end{proof}

\begin{lemma}\label{lemma-U-2}
Let $1< p< \infty$.
Then 
$$
A_p(\e, \eta) =A_{p^\prime}(\e, \eta), \quad
B_p(\e, \eta) =C_{p^\prime} (\e, \eta), \quad \text{ and } \quad
D_{p}  (\e, \eta)= D_{p^\prime} (\e, \eta),
$$
where $p^\prime =\frac{p}{p-1}$.
\end{lemma}

\begin{proof}
This follows by a duality argument. Indeed, 
let $u, w\in W_0^{1, 2} (\omega_{\e, \eta})$ be weak solutions of
$-\Delta u=F+ \text{\rm div}(f)$ in $\omega_{\e, \eta}$ and
$-\Delta w=G+\text{\rm div}(g)$ in $\omega_{\e, \eta}$, respectively,
where $f\in L^p(\omega_{\e, \eta}; \R^d)\cap L^2(\omega_{\e, \eta}; \R^d)$,
$g\in L^{p^\prime} (\omega_{\e, \eta}; \R^d)\cap L^2(\omega_{\e, \eta}; \R^d)$,
$F \in L^p(\omega_{\e, \eta})\cap L^2(\omega_{\e, \eta})$ and
$G \in L^{p^\prime}(\omega_{\e, \eta})\cap L^2(\omega_{\e, \eta})$.
Then
\begin{equation}
\aligned
\int_{\omega_{\e, \eta}} u\,  G\, dx-\int_{\omega_\e, \eta} \nabla u\cdot g \, dx
& =\int_{\omega_{\e, \eta}} \nabla u \cdot \nabla w\, dx\\
&=\int_{\omega_{\e, \eta}} w\,  F \, dx -\int_{\omega_\e, \eta} \nabla w \cdot f\, dx.
\endaligned
\end{equation}
\end{proof}

\begin{thm}\label{thm-D}
Let $C_p(\e, \eta)$ and $D_p (\e, \eta)$ be defined by \eqref{C-p} and  \eqref{D-p}, respectively. 

\begin{enumerate}

\item

For $2\le p< \infty$ and $\e, \eta\in (0, 1]$,
\begin{equation}\label{C-p-1}
C_p (\e, \eta)
\le 
\left\{
\aligned
& C \, \e  \eta^{\frac{2-d}{2}} & \quad & \text{ for } d\ge 3,\\
& C \, \e |\ln (\eta/2)|^{\frac12} & \quad & \text{ for } d=2.
\endaligned
\right.
\end{equation}

\item

For $1< p< \infty$ and $\e, \eta\in (0, 1]$,
\begin{equation}\label{D-p-1}
D_p (\e, \eta) \le 
\left\{
\aligned
& C \, \e^2  \eta^{2-d} & \quad & \text{ for } d\ge 3,\\
& C \, \e^2 |\ln (\eta/2)| & \quad & \text{ for } d=2.
\endaligned
\right.
\end{equation}
\end{enumerate}
The constants  $C$ depend only on $d$, $p$ and $c_0$ in \eqref{condition-0}.
\end{thm}

\begin{proof}
The estimates for $2\le p< \infty$ in \eqref{C-p-1} and \eqref{D-p-1} are given by Theorem \ref{thm-3.0}.
The case $1< p<2$  for $D_p(\e, \eta)$ follows from the case $2<p<\infty$ by Lemma \ref{lemma-U-2}.
\end{proof}

In the remaining of this section  we shall show that  the estimates in Theorem \ref{thm-D}  are sharp.
To this end, we consider  a periodically perforated domain,
\begin{equation}\label{p-domain}
\omega_{\e, \eta} 
=\R^d \setminus \bigcup_{k \in \mathbb{Z}^d} \e (k + \eta T),
\end{equation}
where $T$ is the closure of a bounded Lipschitz subdomain of $Y$.
We assume that  $Y\setminus T$ is connected, $B(0, c_0) \subset T$, and that 
dist$(\partial T, \partial Y) \ge c_0>0$.
Let $\chi_\eta$ be a $Y$-periodic function in $H^1_{loc}(\R^d)$ such that
\begin{equation}\label{chi-1}
-\Delta \chi_\eta =\eta^{d-2} \quad \text{ in } \omega_{1, \eta}
\quad \text{ and }
\quad
\chi_\eta =0 \quad \text{ on } \R^d \setminus \omega_{1, \eta},
\end{equation}
where $\eta \in (0, 1]$.
The existence and uniqueness of $\chi_\eta$ may be proved by using the Lax-Milgram Theorem on the Hilbert space
$\{ u\in H^1_{per} (Y): u=0 \text{ in } \eta T \}$, where $H^1_{per}(Y)$ denotes the closure of the set of
smooth $Y$-periodic functions in $H^1(Y)$.

\begin{lemma}\label{lemma-chi-1}
Let $\chi_\eta$ be the $Y$-periodic function defined by \eqref{chi-1}.
Then, if $d\ge 3$,
\begin{equation}\label{chi-2}
 \int_Y \chi_\eta \, dx  \approx 1 
 \quad \text{ and } \quad
 \left(\int_Y |\nabla \chi_\eta|^2 \,dx \right)^{1/2} \approx  \eta^{\frac{d-2}{2}}.
\end{equation}
If $d=2$, we have
\begin{equation}\label{chi-3}
 \int_Y \chi_\eta \, dx  \approx |\ln (\eta/2)|
 \quad \text{ and } \quad
 \left(\int_Y |\nabla \chi_\eta|^2 \,dx \right)^{1/2} \approx  |\ln (\eta/2)|^{\frac12}.
\end{equation}
\end{lemma}

\begin{proof}
Suppose $d\ge 3$.
Using
\begin{equation}\label{chi-4}
\int_Y |\nabla \chi_\eta|^2\, dx
=\eta^{d-2} \int_Y \chi_\eta\, dx,
\end{equation}
we obtain
$$
\int_Y |\nabla \chi_\eta|^2\, dx
\le \eta^{d-2}  \left(\int_Y |\chi_\eta|^2 \, dx \right)^{1/2}
\le C \eta^{\frac{d-2}{2}} \left(\int_Y |\nabla \chi_\eta|^2\, dx \right)^{1/2},
$$
where we have used Poincar\'e inequality \eqref{P} for $p=2$.
It follows that
$$
\left(\int_Y |\nabla \chi_\eta|^2\, dx \right)^{1/2}
\le C \eta^{\frac{d-2}{2}}
\quad \text{ and } \quad
 \int_Y \chi_\eta\, dx\le C,
$$
where $C$ depends only on $d$ and $c_0$.
To prove the reverse inequalities, we construct a function $\psi\in H^1_{per}(Y)$ such that $0\le \psi \le1$ in $Y$,
$\psi=1$ in $Y\setminus C \eta T$, 
$\psi=0$ on $\eta T$, and $|\nabla \psi |\le C \eta^{-1}$.
Note that
$$
\int_Y \psi\, dx \approx 1 \quad \text{ and } \quad
\left(\int_Y |\nabla \psi|^2\, dx \right)^{1/2} \le C \eta^{\frac{d-2}{2}}.
$$
This yields
$$
\aligned
\eta^{d-2}
\int_Y \psi\, dx
 & =\int_Y \nabla \chi_\eta \cdot \nabla \psi\, dx\\
&\le C  \eta^{\frac{d-2}{2}}\| \nabla \chi_\eta\|_{L^2(Y)}.
\endaligned
$$
Hence, $ \eta^{\frac{d-2}{2}} \le  C\| \nabla \chi_\eta \|_{L^2(Y)}$, which, together with \eqref{chi-4}, also gives
$\int_Y \chi_\eta\, dx  \ge c>0$.

Consider the case $d=2$.
The proof for the upper bounds is the same.
However, the same choice of $\psi$ only produces lower bounds for $\int_Y \chi_\eta \, dx$ and
$\|\nabla \chi_\eta  \|_{L^2(Y)}$ by a positive constant $c$.
For optimal lower bounds, we may assume  that $\eta$ is sufficiently small and 
$\eta T \subset B(0, C_0 \eta) \subset B(0, 1/4)$.
Let $\psi $ be the function in $H^1_{per}(Y)$ such that 
\begin{equation}\label{chi-5}
\left\{
\aligned 
 \psi  & =1-\frac{\ln (1/2)}{\ln (C_0\eta)} & \quad & \text{ in }  Y\setminus B(0, 1/2),\\
 \psi (x)  & = 1-\frac{\ln |x|}{\ln (C_0\eta)}  & \quad & \text{ for  } x\in B(0, 1/2) \setminus B(0, C_0\eta),\\
 \psi & =0 & \quad & \text{  in } B(0, C_0\eta).\\
 \endaligned
 \right.
 \end{equation}
 A direct computation shows that
$
\int_Y \psi\, dx  \ge c>0,
$
and
$$
\int_Y |\nabla \psi |^2\, dx= |\ln (C_0\eta)|^{-2}
 \int_{B(0, 1/2)\setminus B(0, C_0\eta)} |x|^{-2} \, dx
 \le C |\ln \eta|^{-1}.
$$
It follows that
$$
c\le \int_Y \psi\, dx
=\int_Y \nabla \chi_\eta \cdot \nabla \psi\, dx
\le C |\ln \eta|^{-\frac12}  \|\nabla \chi_\eta\|_{L^2(Y)}.
$$
As a result, we have proved that  $\|\nabla \chi_\eta\|_{L^2(Y)} \ge c |\ln (\eta/2) |^{1/2}$.
By \eqref{chi-4}, this yields  $ \int_Y \chi_\eta \, dx \ge c |\ln (\eta/2) |$.
\end{proof}

The next theorem shows that the estimate \eqref{D-p-1} for $D_p (\e, \eta)$ is sharp for $d\ge 2$ and $1< p< \infty$.

\begin{thm}\label{thm-Dp}
Let $D_p (\e, \eta)$ be defined by \eqref{D-p} for the periodically perforated domain $\omega_{\e, \eta}$ in \eqref{p-domain}.
Then, for $d\ge 3$,
\begin{equation}\label{D-p-low}
 c\,  \e^2 \eta^{2-d} \le D_p (\e, \eta) \le C \e^2 \eta^{2-d},
\end{equation}
and for $d=2$,
\begin{equation}\label{D-p-low2}
 c\,  \e^2  |\ln (\eta/2)|  \le D_p (\e, \eta) \le C \e^2  |\ln (\eta/2)|,
\end{equation}
where $1<p<\infty$ and $C, c>0$ depend only on $d$, $p$ and $c_0$.
\end{thm}

\begin{proof}
The upper bounds are given by Theorem \ref{thm-D}.
In view of Lemmas \ref{lemma-U-1} and \ref{lemma-U-2},
to prove the lower bounds, we may assume $\e=1$ and $1< p\le 2$.
Fix $R>10$.
Let $\varphi \in C_0^\infty(\R^d) $ such that
$0\le \varphi \le 1$ in $\R^d$,
$\varphi=1$ in $B(0, R)$, $\varphi=0$  outside of $B(0, 2R)$, and
$|\nabla \varphi |\le CR^{-1}$, $|\nabla^2 \varphi |\le C R^{-2}$.
Using
\begin{equation}\label{eq-00}
-\Delta (\chi_\eta \varphi)
=\eta^{d-2} \varphi
-2 \nabla \chi_\eta \cdot \nabla \varphi
- \chi_\eta \Delta \varphi 
\end{equation}
in $\omega_{1, \eta}$ and $\chi_\eta \varphi=0$ on $\partial \omega_{1, \eta}$, we obtain 
$$
\aligned
c R^{\frac{d}{p}} \| \chi_\eta\|_{L^p(Y)}
&\le 
\|\chi_\eta \varphi \|_{L^p(\omega_{1, \eta})}\\
& \le D_p (1, \eta)
\left\{ 
C \eta^{d-2} R^{\frac{d}{p} }
+ C R^{-1+\frac{d}{p}} \|\nabla \chi_\eta \|_{L^p(Y)}
+ C R^{-2 +\frac{d}{p}} \|\chi_\eta\|_{L^p(Y)} \right\},
\endaligned
$$
where we have used the periodicity of $\chi_\eta$.
It follows that
\begin{equation}\label{chi-55}
 \| \chi_\eta\|_{L^1(Y)}\le D_p (1,  \eta)
\left\{
C \eta^{d-2}
+ C R^{-1} \| \nabla \chi_\eta \|_{L^p(Y)}
+ C R^{-2}  \| \chi_\eta \|_{L^p(Y)} \right\}
\end{equation}
for any $R>10$, where we have used the fact that  $\| \chi_\eta\|_{L^p(Y)} \ge \| \chi_\eta\|_{L^1(Y)}$.
By letting $R \to \infty$ in \eqref{chi-55}, we see that 
$$
D_p(1, \eta) \ge c\,  \eta^{2-d} \| \chi_\eta\|_{L^1(Y)}.
$$
In view of Lemma \ref{lemma-chi-1}, this gives the lower bounds in \eqref{D-p-low} and
\eqref{D-p-low2} for $\e=1$.
\end{proof}

The estimate \eqref{C-p-1} for $C_p(\e, \eta)$  is also sharp.

\begin{thm}\label{thm-C-p-1}
Let $C_p (\e, \eta)$ be defined by \eqref{C-p} for the periodically perforated domain $\omega_{\e, \eta}$ in \eqref{p-domain}.
Then, for $d\ge 3$,
\begin{equation}\label{C-p-low}
c\,  \e \eta^{\frac{2-d}{2}}
\le C_p (\e, \eta) \le C\,  \e \eta^{\frac{2-d}{2}},
\end{equation}
and for $d=2$,
\begin{equation}\label{C-p-low2}
c \, \e |\ln (\eta/2)|^{\frac12}
\le C_p (\e, \eta)
\le C\,  \e |\ln (\eta/2)|^{\frac12},
\end{equation}
where $2\le p< \infty$ and $C, c$ depend only on $d$, $p$ and $T$.
\end{thm}

\begin{proof}
The upper bounds are given by Theorem \ref{thm-D}.
To prove the lower bounds, in view of Lemmas \ref{lemma-U-1} and \ref{lemma-U-2},
we assume $\e=1$ and consider $B_p (1, \eta)$ for $1< p\le 2$.

Let $\varphi$ and $\chi_\eta$ be the same as in the proof of Theorem \ref{thm-Dp}.
Then
\begin{equation}\label{chi-6}
\|\nabla (\chi_\eta \varphi)\|_{L^p(\omega_{1, \eta})}
\le B_p(1, \eta)
\left\{ 
C \eta^{d-2} R^{\frac{d}{p}}
+ C R^{-1 +\frac{d}{p}} \|\nabla \chi_\eta\|_{L^p(Y)}
+ C R^{-2 +\frac{d}{p}} \| \chi_\eta\|_{L^p(Y)} \right\}.
\end{equation}
By Sobolev imbedding,  $\|\chi_\eta \varphi\|_{L^q(\R^d)}
\le C \| \nabla (\chi_\eta \varphi)\|_{L^p(\R^d)}$, where
$\frac{1}{q}=\frac{1}{p}-\frac{1}{d}$ and $1<p<d$.
It follows that
\begin{equation}\label{chi-7}
\| \chi_\eta \|_{L^q(Y)}
\le B_p (1, \eta)
\left\{
C \eta^{d-2} R + C \| \nabla \chi_\eta\|_{L^p(Y)}
+ C R^{-1} \| \chi_\eta\|_{L^p(Y)} \right\}
\end{equation}
for any $R\ge 10$.

Suppose that $d\ge 3$ and $1<p\le 2$.
Note that $\|\nabla \chi_\eta\|_{L^p(Y)} \le \| \nabla \chi_\eta \|_{L^2(Y)} \le C \eta^{\frac{d-2}{2}}$ and
$$
\|\chi_\eta\|_{L^p(Y)} \le \| \chi_\eta \|_{L^2(Y)}
\le C  \|\nabla \chi_\eta \|_{L^2(Y)} \eta^{\frac{2-d}{2}}\le C.
$$
By letting $R= C \eta^{\frac{2-d}{2}}$ in \eqref{chi-7}, we obtain 
$$
C_{p^\prime} (1, \eta)=B_p(1, \eta) \ge c\,  \eta^{\frac{2-d}{2}}
$$
for $1<p\le 2$. This gives the lower bound in \eqref{C-p-low}.
In the case $d=2$ and $1<p<2$,
we choose $R= C|\ln (\eta/2)|^{1/2}$ in \eqref{chi-7}.
In view of \eqref{chi-3} we obtain 
\begin{equation}\label{chi-9}
B_p(1, \eta) \ge c\,  |\ln (\eta/2)|^{1/2}.
\end{equation}

Finally,  note that by 
 \eqref{chi-6}, we have 
$$
\|\nabla \chi_\eta \|_{L^p(Y)}
\le B_p (1, \eta)
\left\{  C \eta^{d-2}
+ C R^{-1} \|\nabla \chi_\eta\|_{L^p(Y)}
+ C R^{-2} \| \chi_\eta\|_{L^p(Y)}
\right\}.
$$
By letting $R\to \infty$, we obtain 
\begin{equation}\label{chi-8}
B_p (1, \eta) \ge c\,  \eta^{2-d} \|\nabla \chi_\eta\|_{L^p(Y)}
\end{equation}
for $d\ge 2$ and $1< p< \infty$.
As a result, we see that \eqref{chi-9} also holds for $d=2$ and $p=2$.
\end{proof}



\section{Estimates of $\nabla u$}\label{section-GU}

In this section we establish upper and lower bounds for $A_p(\e, \eta)$ and $B_p(\e, \eta)$, defined by \eqref{A-p}
and \eqref{B-p}, respectively.
Throughout this section we assume $\{T_k\}$ are the closures of open subsets of $Y$ with uniform $C^1$ boundaries and satisfy
\eqref{condition-0}. 

\begin{thm}\label{thm-A-p}
Let $\omega_{\e, \eta}$ be given by \eqref{omega}.
Let $u\in W_0^{1, 2} (\omega_{\e, \eta})$ be  a weak solution  to the
Dirichlet problem:  $ -\Delta u=\text{\rm div}(f)$ in $\omega_{\e, \eta}$ and
$u=0$ on $\partial\omega_{\e, \eta}$, 
where $f\in L^2(\omega_{\e, \eta}; \R^d) \cap L^p(\omega_{\e, \eta}; \R^d)$ for some $1<p<\infty$.
Then
\begin{equation}\label{Ap0}
\|\nabla u_\e \|_{L^p(\omega_{\e, \eta})}
\le C_{\delta}  \eta^{-d|\frac12 -\frac{1}{p}|-\delta}
 \| f\|_{L^p(\omega_{\e, \eta})}
\end{equation}
for any $\delta\in (0, 1)$,
where $C_{ \delta}$ depends on $d$, $p$, $\delta$ and $\{T_k\}$.
\end{thm}

\begin{proof}
By rescaling and duality we may assume $\e=1$ and $p>2$.
It follows by Lemma \ref{local-lemma-1} that  for each $k \in \mathbb{Z}^d$,
$$
\int_{k+ (Y\setminus \eta T_k)} |\nabla u|^p\, dx
\le C \left\{ \eta^{-p} \int_{k + (\tY\setminus \eta T_k)} |u|^p\, dx
+\int_{k+ (\tY\setminus \eta T_k)} |f|^p\, dx \right\},
$$
where $C$ depends on $d$, $p$, $c_0$, and the uniform $C^1$ characters of $\{T_k\}$.
By summing the inequalities above over $k\in \mathbb{Z}^d$,
 we obtain 
\begin{equation}\label{local-k}
\int_{\omega_{1, \eta}}
|\nabla u|^p\, dx
\le C \left\{ \eta^{-p} \int_{\omega_{1, \eta}} |u|^p\, dx
+\int_{\omega_{1, \eta}} |f|^p\, dx \right\}.
\end{equation}
This, together with \eqref{g-3-0}, gives
\begin{equation}\label{Ap1}
\| \nabla u\|_{L^p(\omega_{1, \eta})}
\le C_p \eta^{-\frac{d}{2}} \| f\|_{L^p(\omega_{1, \eta})}
\end{equation}
for any $p>2$ and  $d\ge 3$.
In the case $d=2$, we may  use \eqref{local-k} and \eqref{g-3-1} to obtain 
\begin{equation}\label{Ap2}
\| \nabla u\|_{L^p(\omega_{1, \eta})}
\le C \eta^{-1} |\ln (\eta/2)|^{1/2} 
 \| f\|_{L^p(\omega_{1, \eta})}
\end{equation}
for $p>2$.
Note that for $p=2$, we have the energy estimate,
\begin{equation}\label{AP3}
\| \nabla u\|_{L^2(\omega_{1, \eta})}
\le \| f\|_{L^2(\omega_{1, \eta})}.
\end{equation}
Thus, by Riesz-Thorin Interpolation Theorem, if $d\ge 3$ and $2<p<p_0$,
\begin{equation}\label{Ap5}
\|\nabla u\|_{L^p(\omega_{1, \eta})}
\le C \eta^\alpha \| f\|_{L^p (\omega_{1, \eta})},
\end{equation}
where
$$
\alpha = -\frac{d}{2} \cdot \frac{\frac12 -\frac{1}{p}}{ \frac12 -\frac{1}{p_0}}.
$$
Since
$$
\alpha \to -d \left(\frac12-\frac{1}{p}\right) \quad \text{ as } p_0 \to \infty,
$$
for any $\delta\in (0, 1)$, there exists $p_0>p$ such that
$$
\alpha > -d \left(\frac12-\frac{1}{p}\right)-\delta.
$$
As a result, we have proved that if $d\ge 3$,
\begin{equation}\label{estimate-4.10}
\| \nabla u\|_{L^p(\omega_{1,\delta})}
\le C_{ \delta} \, \eta^{-d \left|\frac{1}{2}-\frac{1}{p} \right|-\delta } \| f\|_{L^p(\omega_{1, \eta})}
\end{equation}
 for any $p>2$ and $\delta\in (0, 1)$.
 Using \eqref{Ap2}, the same argument also yields \eqref{Ap0} for $p>2$ in the case $d=2$.
 The logarithmic factor can be absorbed into $\eta^{-\delta}$.
\end{proof}

\begin{thm}\label{thm-B-p}
Let $\omega_{\e, \eta}$ be the same perforated domain as in Theorem \ref{thm-A-p}.
Let $u\in W^{1, 2}_0(\omega_{\e, \eta})$ be  a weak solution  to the
Dirichlet problem:  $ -\Delta u=F$ in $\omega_{\e, \eta}$ and
$u=0$ on $\partial\omega_{\e, \eta}$, 
where $F\in L^2(\omega_{\e, \eta}) \cap L^p(\omega_{\e, \eta})$ for some $1<p<\infty$.
Then, if $d\ge 3$,
\begin{equation} \label{Bp0}
\| \nabla u\|_{L^p(\omega_{\e, \eta})}
\le \left\{
\aligned
& C\,  \e \eta^{1-\frac{d}{2}}\| F\|_{L^p(\omega_{\e, \eta})}
 & \quad & \text{ if } 1< p\le 2, \\
&C_{ \delta}\, \e  \eta^{1-d+\frac{d}{p}-\delta }\| F\|_{L^p(\omega_{\e, \eta})}
 & \quad & \text{ if } 2< p< \infty,
\endaligned
\right.
\end{equation}
for any $\delta\in (0, 1)$,
where $C_{\delta}$ depends on $d$, $p$, $\delta$ and $\{T_k\}$.
If $d=2$, one has
\begin{equation} \label{Bp1}
\| \nabla u\|_{L^p(\omega_{\e, \eta})}
\le \left\{
\aligned
& C\,  \e |\ln (\eta/2)|^{\frac12}\| F\|_{L^p(\omega_{\e, \eta})}
 & \quad & \text{ if } 1< p\le 2, \\
&C_{ \delta}\, \e  \eta^{-1+\frac{2}{p}-\delta }\| F\|_{L^p(\omega_{\e, \eta})}
 & \quad & \text{ if } 2< p< \infty,
\endaligned
\right.
\end{equation}
for any $\delta\in (0, 1)$.
\end{thm}

\begin{proof}
Since $B_p (\e, \eta)=C_{p^\prime}(\e, \eta)$ for $1< p< \infty$,
the estimates \eqref{Bp0} and \eqref{Bp1} for the case $1< p\le 2$ are given by Theorem \ref{thm-D}.

To treat  the case $2<p<\infty$, we may assume $\e=1$ by rescaling.
Suppose  $d\ge 3$.
Let $u\in W^{1, 2}_0(\omega_{1, \eta})$ be a weak solution of
$-\Delta u=F$ in $\omega_{1, \eta}$.
Note that by Lemma \ref{local-lemma-1},   for each $k \in \mathbb{Z}^d$,
\begin{equation}\label{Bp2}
\int_{k+ (Y\setminus \eta T_k)} |\nabla u|^p\, dx
\le C \left\{ \eta^{-p} \int_{k + (\tY\setminus \eta T_k)} |u|^p\, dx
+\int_{k+ (\tY\setminus \eta T_k)} |F|^p\, dx \right\},
\end{equation}
where $C$ depends on $d$, $p$, $c_0$, and the uniform $C^1$ characters of $\{T_k\}$.
It follows by summation that
$$
\aligned
\| \nabla u \|_{L^p(\omega_{1, \eta})}
& \le C \left\{ \eta^{-1} \| u \|_{L^p(\omega_{1,\eta})}
+ \| F \|_{L^p(\omega_{1, \eta})} \right\}\\
& \le C \eta^{1-d} \| F\|_{L^p(\omega_{1, \eta})}
\endaligned
$$
for any $p>2$,
where we have used \eqref{g-3-0} for the last inequality.
By Riesz-Thorin Interpolation Theorem, this, together with the $L^2$ estimate
$$
\|\nabla u \|_{L^2(\omega_{1, \eta})} \le C \eta^{1-\frac{d}{2}} \| F\|_{L^2(\omega_{\e, \eta})},
$$
yields 
$$
\|\nabla u\|_{L^p(\omega_{1, \eta})}
\le C \eta^\alpha
\| F \|_{L^p(\omega_{1, \eta})},
$$
where $2<p< p_0$ and
$$
\alpha =1-d +\frac{d}{2} \cdot \frac{\frac{1}{p}-\frac{1}{p_0}}{\frac12 -\frac{1}{p_0}}.
$$
Since $\alpha \to 1-d +\frac{d}{p}$ as $p_0\to \infty$, by choosing $p_0$ sufficiently large,  we obtain 
$$
\| \nabla u \|_{L^p(\omega_{1, \eta})}
\le C_{ \delta} \eta^{1-d +\frac{d}{p} -\delta} \| F \|_{L^p(\omega_{1, \eta})}
$$
for any $\delta \in (0, 1)$.

The proof for the case $d=2$ and $2<p<\infty$ is similar.
It follows by \eqref{Bp2} and \eqref{g-3-1} that
$$
\|\nabla u \|_{L^p(\omega_{1, \eta})}
\le C \eta^{-1} |\ln (\eta/2)| \| F \|_{L^p(\omega_{1, \eta})},
$$
for any $p>2$. This, together with the $L^2$ estimate 
$$
\|\nabla u \|_{L^2(\omega_{1, \eta})}
\le C |\ln (\eta/2)|^{1/2} \| F \|_{L^2(\omega_{1, \eta})},
$$
yields the desired estimate by interpolation.
\end{proof}

Next we show that 
the estimates in Theorems \ref{thm-A-p} and \ref{thm-B-p} are near optimal.

\begin{lemma}\label{lemma-chi01}
Let $\chi_\eta$ be the $Y$-periodic function defined by \eqref{chi-1}.
Let $2<p< \infty$.
Then, if $d\ge 3$,
\begin{equation}\label{chi-01}
\left(\int_Y |\chi_\eta|^p\, dx \right)^{1/p}
\le C \quad \text{ and }
\quad
\left(\int_Y |\nabla \chi_\eta|^p\, dx \right)^{1/p}
\ge c\, \eta^{\frac{d}{p}-1}.
\end{equation}
If $d=2$, we have
\begin{equation}\label{chi-02}
\left(\int_Y |\chi_\eta|^p \, dx \right)^{1/p}
\le C |\ln (\eta/2)|
\quad \text{ and }
\quad
\left(\int_Y |\nabla \chi_\eta|^p\, dx \right)^{1/p}
\ge c \, \eta^{\frac{2}{p}-1}.
\end{equation}
\end{lemma}

\begin{proof}
Let $\varphi$ be a cut-off function in $C_0^\infty(\R^d)$ such that
$0\le \varphi \le 1$ in $\R^d$,
$\varphi=1$ in $B(0, R)$,
$\varphi=0$ outside of $B(0, 2R)$,
$|\nabla \varphi|\le C R^{-1}$ and
$|\nabla^2 \varphi|\le CR^{-2}$.
Since $\chi_\eta \varphi=0$ on $\partial \omega_{1, \eta}$ and
$$
-\Delta (\chi_\eta \varphi)= \eta^{d-2} \varphi
-2 \nabla \chi_\eta \cdot \nabla \varphi - \chi_\eta \Delta \varphi
$$
in $\omega_{1, \eta}$, it follows from \eqref{g-3-0} that if $d\ge 3$,
$$
\| \chi_\eta \varphi \|_{L^p(\omega_{1, \eta})}
\le C_p \eta^{2-d}
\left\{ \eta^{d-2} R^{\frac{d}{p}}
+ R^{-1 +\frac{d}{p}} \|\nabla \chi_\eta \|_{L^p(Y)}
+ R^{-2 +\frac{d}{p}} \|\chi_\eta \|_{L^p(Y)}
\right\}.
$$
This yields
$$
\| \chi_\eta\|_{L^p(Y)}
\le C_p \eta^{2-d}
\left\{ \eta^{d-2}
+ R^{-1} \| \nabla \chi_\eta\|_{L^p(Y)} + R^{-2} \| \chi_\eta\|_{L^p(Y)} \right\}.
$$
By letting $R\to \infty$ we obtain $\|\chi_\eta \|_{L^p(Y)} \le C$ for $d\ge 3$ and $2<p< \infty$.
If $d=2$, we may use \eqref{g-3-1} in the place of \eqref{g-3-0} to obtain 
$\|\chi_\eta\|_{L^p(Y)} \le C |\ln (\eta/2)|$ for $2< p< \infty$.

To prove the lower bounds for $\|\nabla \chi_\eta \|_{L^p(Y)}$,
we construct a function $\psi\in H_{per}^1(Y)$ as in the proof of Lemma \ref{lemma-chi-1}.
Note that 
$$
\int_Y \psi\, dx \approx 1
\quad \text{ and }
\quad
\left(\int_Y |\nabla \psi |^{p^\prime}\, dx \right)^{1/p^\prime}
\le C \eta^{-1+\frac{d}{p^\prime}}.
$$
It follows that
$$
C \eta^{d-2}
\le \| \nabla \chi_\eta \|_{L^p(Y)}
\| \nabla \psi \|_{L^{p^\prime}(Y)}.
$$
This yields 
$$
\| \nabla \chi_\eta\|_{L^p(Y)} \ge c\,  \eta^{\frac{d}{p}- 1}
$$
for $d\ge 2$.
\end{proof}

The following theorem provides lower bounds for $A_p(\e, \eta)$.

\begin{thm}\label{thm-Ap000}
Let $A_p(\e, \eta)$ be defined by \eqref{A-p} for the periodically perforated domain 
$\omega_{\e, \eta}$ in \eqref{p-domain}. Then, for $d\ge 3$,
\begin{equation}\label{Ap10}
A_p (\e, \eta) \ge c\,  \eta^{-d |\frac{1}{2}-\frac{1}{p}| },
\end{equation}
and for $d=2$,
\begin{equation}\label{Ap11}
A_p (\e, \eta) \ge c\,   |\ln (\eta/2)|^{-\frac12} \eta^{-2 |\frac12-\frac{1}{p}|},
\end{equation}
where $1<p< \infty$ and $c>0$ depends only on $d$, $p$ and $T$.
\end{thm}

\begin{proof}
In view of Lemmas \ref{lemma-U-1} and \ref{lemma-U-2}
we may assume that $\e=1$ and $2\le p< \infty$.
We first consider the case $d\ge 3$.
Let $u$ be a weak solution of 
$-\Delta u=\text{\rm div}(f) $ in $\omega_{1, \eta}$  with
$u=0$ on $\partial\omega_{1, \eta}$.
By Sobolev imbedding, we have
$$
\| u\|_{L^{q^\prime} (\omega_{1, \eta})}
\le C  \| \nabla u\|_{L^{p^\prime} (\omega_{1, \eta})}
\le C A_p (1, \eta) \| f\|_{L^{p^\prime} (\omega_{1, \eta})},
$$
where $\frac{1}{q^\prime }=\frac{1}{p^\prime }-\frac{1}{d}$ and $1<p^\prime < d$.
By duality this implies that if $-\Delta v= G$ in $\omega_{1, \eta}$ and $v=0$ on $\partial\omega_{1, \eta}$,
then
$$
\| \nabla v \|_{L^{p} (\omega_{1, \eta})}
\le C A_p (1, \eta) \| G \|_{L^{q} (\omega_{1, \eta})}.
$$
As a result, we have proved that if $-\Delta u=F +\text{\rm div}(f)$ in $\omega_{1, \eta}$ and $u=0$ on $\partial\omega_{1, \eta}$, then
\begin{equation}\label{o-p-1}
\| \nabla u\|_{L^p(\omega_{1, \eta})}
\le C A_p(1, \eta)
\big\{ \| F \|_{L^q(\omega_{1, \eta})} + \| f\|_{L^p(\omega_{1, \eta})} \big\},
\end{equation}
where $d^\prime< p< \infty$ and $\frac{1}{q}=\frac{1}{p} +\frac{1}{d}$.

Let $\varphi \in C_0^\infty(( B(0, 2R))$ be a cut-off function such that $\varphi=1$ in $B(0, R)$, $0\le \varphi \le 1$,
$|\nabla \varphi |\le C/R$, and $|\nabla^2 \varphi |\le C/R^2$, where $R\ge d$.
Let $\chi_\eta$ be the 1-periodic function in Lemma \ref{lemma-chi01}.
Since
\begin{equation}\label{eq-0}
-\Delta (\chi_\eta \varphi)
=\eta^{d-2} \varphi -2\,  \text{\rm div}( \chi_\eta \nabla \varphi) + \chi_\eta \Delta \varphi
\end{equation}
in $\omega_{1, \eta}$ and $\chi_\eta \varphi=0$ on $\partial \omega_{1, \eta}$,
we deduce from \eqref{o-p-1} that 
$$
\aligned
\| \nabla (\chi_\eta \varphi)\|_{L^p(\omega_{1, \eta})}
 &\le C  A_p(1, \eta)
\left\{
\eta^{d-2} R^{\frac{d}{q}}
+ R^{-1+\frac{d}{p}} \| \chi_\eta \|_{L^p(Y)}
+ R^{-2 +\frac{d}{q}} \| \chi_\eta \|_{L^q(Y)}  \right\}\\
&\le C A_p (1, \eta)
\big\{ \eta^{d-2} R^{\frac{d}{p}+1}
+ R^{-1 +\frac{d}{p}}
\big\},
\endaligned
$$
where we have used \eqref{chi-01}.
Since $\chi_\eta$ is $Y$-periodic, it is not hard to see that
$$
\| \nabla (\chi_\eta \varphi) \|_{L^p(\omega_{1, \eta})}
\ge \| \nabla \chi_\eta \|_{L^p (B(0, R))}
\approx
R^{\frac{d}{p} }\| \nabla \chi_\eta \|_{L^p(Y)}.
$$
It follows that 
$$
\aligned
\|\nabla \chi_\eta\|_{L^p(Y)}
 & \le C A_p (1, \eta)
\big\{ \eta^{d-2} R + R^{-1}\big\}\\
&\le C A_p (1, \eta) \eta^{\frac{d}{2}-1},
\endaligned
$$
where we have chosen $ R= C\eta^{-\frac{d-2}{2}}$.
This, together with \eqref{chi-01},  yields   \eqref{Ap10} for $2\le p< \infty$ and $\e=1$.

The proof for the case $d=2$ and $2<p<\infty$ is similar.
 Indeed,  
using \eqref{chi-02},  we may show that
$$
\| \nabla (\chi_\eta \varphi) \|_{L^p(\omega_{1, \eta})}
\le C A_p (1, \eta)
\big\{R^{\frac{2}{p} +1 } 
+   | \ln (\eta/2) | R^{\frac{2}{p}-1} \big\}.
$$
It follows that
$$
\aligned
\|\nabla \chi_\eta\|_{L^p(Y)}
 & \le C A_p (1, \eta) \big\{
R +  |\ln( \eta/2)| R^{-1} \big\}\\
& \le C A_p (1, \eta) |\ln (\eta/2)|^{1/2},
\endaligned
$$
where we have let $R=C |\ln (\eta/2)|^{1/2}$.
This, together with \eqref{chi-02},  gives \eqref{Ap11} for $2< p< \infty$ and $\e=1$.

Finally, in the case $d=2$ and $p=2$, we note that by Lemma \ref{lemma-3.1},
$$
C_2 (1, \eta) \le C A_2 (1, \eta) |\ln (\eta/2)|^{1/2}.
$$
By Theorem \ref{thm-C-p-1}, $C_2(1, \eta) \ge c |\ln (\eta/2)|^{1/2}$.
It follows that $A_2(1, \eta)\ge c>0$.
\end{proof}

The next theorem gives lower bounds for $B_p(\e, \eta)$.

\begin{thm}\label{thm-Bp0}
Let $B_p(\e, \eta)$ be defined by \eqref{B-p} for the periodically perforated domain $\omega_{\e, \eta}$ in 
\eqref{p-domain}.
Then, for $d\ge 3$,
\begin{equation}\label{Bp01}
B_p (\e, \eta)
\ge \left\{
\aligned
&  c\, \e  \eta^{1-\frac{d}{2}} & \quad & \text{ for } 1<p\le 2,\\
& c\, \e \eta^{1-d +\frac{d}{p}}  & \quad & \text{ for } 2<p< \infty,
\endaligned
\right.
\end{equation}
where $c>0$ depends only on $d$, $p$ and $T$.
If $d=2$, we have
\begin{equation}\label{Bp02}
B_p (\e, \eta)
\ge \left\{
\aligned
&  c\,  \e |\ln (\eta/2)|^{\frac12} & \quad & \text{ for } 1<p\le 2,\\
& c \, \e \eta^{-1+\frac{2}{p}}  & \quad & \text{ for } 2<p< \infty.
\endaligned
\right.
\end{equation}
\end{thm}

\begin{proof}
Since $B_p (\e, \eta)=C_{p^\prime} (\e, \eta)$,
the case $1<p\le 2$ is contained in Theorem \ref{thm-C-p-1}.
To treat  the case $2< p<\infty$, we assume $\e=1$.
As in the proof of Theorem \ref{thm-Dp}, we may use the equation \eqref{eq-00}
to deduce that
$$
\|\nabla \chi_\eta\|_{L^p(Y)}
\le C \eta^{d-2} B_p (1, \eta).
$$
By Lemma \ref{lemma-chi01} this leads to
$$
B_p (1, \eta)
\ge c\,  \eta^{2-d} \| \nabla \chi_\eta\|_{L^p(Y)}
\ge  c\, \eta^{1-d +\frac{d}{p}}
$$
for $d\ge 2$.
\end{proof}

We are now in a position to give  the proof of Theorem \ref{main-thm-2}. 

\begin{proof}[Proof of Theorem \ref{main-thm-2}]
Let $d\ge 3$ and  $1<p<\infty$.
It follows from Theorems  \ref{thm-D}, \ref{thm-A-p} and \ref{thm-B-p} that if $f\in L^p(\omega_{\e, \eta}; \R^d)\cap L^2(\omega_{\e, \eta}; \R^d)$ and
$F \in L^p(\omega_{\e, \eta})\cap L^2 (\omega_{\e, \eta})$,
the weak solution of \eqref{D-omega} in $W^{1, 2}_0(\omega_{\e, \eta})$ satisfies the estimates
\eqref{m-e-1} and \eqref{m-e-2}.
By a density argument this implies that for any $f\in L^p(\omega_{\e, \eta}; \R^d)$ and
$F\in L^p(\omega_{\e, \eta})$, the Dirichlet problem \eqref{D-omega} has a solution $u$ in $W^{1, p}_0(\omega_{\e, \eta})$,
which satisfies the estimates \eqref{m-e-1} and \eqref{m-e-2}.

To prove the uniqueness, let's assume that  $\Delta u=0$ in $\omega_{1, \eta}$
and $u=0$ on $\partial \omega_{1, \eta}$  for some $u\in W^{1, p}(\omega_{1, \eta})$ and $p>1$.
By $L^\infty$ estimates for harmonic functions in Lipschitz domains,
$$
\max_{k + (Y\setminus \eta T_k)} |u|
\le C_\eta \int_{k + (\tY\setminus \eta T_k)} |u|\,dx,
$$
where $C_\eta$ depends on $\eta$, but not on $k$.
Since $u\in L^p(\omega_{1, \eta})$, we see  that $u\in L^\infty(\omega_{1, \eta})$ and that
$
u (x) \to 0 \text{ as } |x| \to \infty.
$
By applying  the maximum principle in $B(0, R) \cap \omega_{\e, \eta}$ and then letting $R \to \infty$,
 we conclude that $u \equiv 0$ in $\omega_{1, \eta}$.
This completes the proof of Theorem \ref{main-thm-2}.
\end{proof}

Theorem \ref{main-thm-2} treats the case $d\ge 3$.
The estimates for $d=2$ are summarized below.

\begin{thm}\label{thm-2d-0}
Suppose  $d=2$ and $1<p<\infty$.
Let $\omega_{\e, \eta}$ be the same as in Theorem \ref{main-thm-2}.
Then, for any $F\in L^p(\omega_{\e, \eta})$ and $f\in L^p(\omega_{\e, \eta}; \R^d)$, the Dirichlet problem 
\eqref{D-omega} has a unique solution in $W^{1, p}_0(\Omega_{\e, \eta})$.
Moreover, the solution satisfies the estimates,
\begin{equation}\label{m-2d-0}
\|\nabla u \|_{L^p(\omega_{\e, \eta})}
\le \left\{
\aligned
& C \e |\ln (\eta/2)|^{\frac{1}{2}} \| F \|_{L^p(\omega_{\e, \eta})}
+ C_\delta \, \eta^{-2 |\frac12-\frac{1}{p}|-\delta} \| f\|_{L^p(\omega_{\e, \eta})} & \quad & \text{ for } 1< p\le 2,\\
& C_\delta \, \e \eta^{-1 +\frac{2}{p}-\delta}
\| F \|_{L^p(\omega_{\e, \eta})} + C_\delta \eta^{-2 |\frac12 -\frac{1}{p}| -\delta } \| f\|_{L^p(\omega_{\e, \eta})}
& \quad & \text{ for } 2< p < \infty,
\endaligned
\right.
\end{equation}
and
\begin{equation}\label{m-2d-1}
\| u \|_{L^p(\omega_{\e, \eta})}
\le \left\{
\aligned
& C \e^2  |\ln (\eta/2)| \| F \|_{L^p(\omega_{\e, \eta})}
+ C_\delta\, \e \eta^{1-\frac{2}{p} -\delta} \| f\|_{L^p(\omega_{\e, \eta})} & \quad & \text{ for } 1< p< 2,\\
& C \e^2 |\ln (\eta/2)| 
\| F \|_{L^p(\omega_{\e, \eta})} + C \e |\ln (\eta/2)|^{\frac12}
 \| f\|_{L^p(\omega_{\e, \eta})}
& \quad & \text{ for } 2\le  p < \infty,
\endaligned
\right.
\end{equation}
for any $\delta \in (0, 1)$, where $C$ depends on $p$ and $\{T_k\}$, and $C_\delta$ also depends on $\delta$.
\end{thm}

\begin{proof}
See \eqref{Ap0} and \eqref{Bp1} for \eqref{m-2d-0}, and  \eqref{C-p-1}-\eqref{D-p-1} for \eqref{m-2d-1}.
\end{proof}




\section{Estimates in a bounded perforated domain}\label{section-B}

In this section we consider the Dirichlet problem \eqref{D-01}
in a bounded perforated domain $\Omega_{\e,\eta}$ given by   \eqref{Omega}.
Throughout this section, unless indicated otherwise,
 we assume that $\Omega$ is a bounded $C^1$ domain and that 
$\{ T_k: \, k \in \mathbb{Z}^d \} $ are the closures of bounded sub-domains of $Y$ with uniform $C^1$ boundaries.

For $\e, \eta\in (0, 1]$,
let $A_p(\e, \eta)$, $B_p(\e, \eta)$, $C_p(\e, \eta)$ and $D_p(\e, \eta)$ be defined by 
\eqref{A-p}, \eqref{B-p}, \eqref{C-p} and \eqref{D-p}, but with $\Omega_{\e, \eta}$ in the place of $\omega_{\e, \eta}$.

\begin{lemma}\label{lemma-b}
Let $\Omega_{\e, \eta}$ be given by \eqref{Omega} and $1< p< \infty$.
For any $f\in L^p(\Omega_{\e, \eta}; \R^d)$ and $F\in L^p(\Omega_{\e, \eta})$, 
the Dirichlet problem \eqref{D-01} has a unique solution in $W^{1, p}_0(\Omega)$.
Moreover, the solution satisfies the estimates,
\begin{equation}\label{b1}
\|\nabla u\|_{L^p(\Omega_{\e, \eta})}
\le A_p (\e, \eta) \| f\|_{L^p(\Omega_{\e, \eta})}
+ B_p (\e, \eta) \| F\|_{L^p(\Omega_{\e, \eta})},
\end{equation}
and
\begin{equation}\label{b2}
\|  u\|_{L^p(\Omega_{\e, \eta})}
\le C_p (\e, \eta) \| f\|_{L^p(\Omega_{\e, \eta})}
+ D_p (\e, \eta) \| F\|_{L^p(\Omega_{\e, \eta})}.
\end{equation}
\end{lemma}

\begin{proof}
Since $\Omega_{\e, \eta}$ is a bounded $C^1$ domain, the existence and uniqueness of solutions for \eqref{D-01}
are given by \cite{JK}.
The estimates \eqref{b1} and \eqref{b2} follow by linearity and a density argument.
\end{proof}

\begin{lemma}\label{lemma-D11}
The three equations  in Lemma \ref{lemma-U-2} continue to hold for the domain  $\Omega_{\e, \eta}$.
\end{lemma}

\begin{proof}
The same duality argument for $ \omega_{\e, \eta}$  works equally well for $\Omega_{\e, \eta}$.
\end{proof}

\begin{thm}\label{thm-D12}
Let $C_p(\e, \eta)$ and $D_p(\e, \eta)$ be defined by $\eqref{C-p}$ and \eqref{D-p}, respectively, but with 
$\Omega_{\e, \eta}$ in the place of $\omega_{\e, \eta}$. Then

\begin{enumerate}

\item

For $2\le  p< \infty$,
\begin{equation}\label{D13}
C_p (\e, \eta) \le 
\left\{ 
\aligned
 & C \min (1, \e \eta^{\frac{2-d}{2}}) & \quad &\text{ if } d\ge 3,\\
 & C \min (1, \e |\ln (\eta/2)|^{1/2}) & \quad & \text{ if } d=2.
 \endaligned
 \right.
 \end{equation}
 
 \item
 
 For  $1< p< \infty$,
 \begin{equation}\label{D14}
 D_p (\e,  \eta)
 \le 
 \left\{
 \aligned
 & C \min (1, \e^2 \eta^{2-d}) & \quad  & \text{ if } d\ge 3,\\
 & C \min (1, \e^2 |\ln (\eta/2)| ) & \quad  & \text{ if } d=2.
 \endaligned
 \right.
 \end{equation}
\end{enumerate}
The constants $C$ depend on $d$, $p$, $c_0$ and $\Omega$.
\end{thm}

\begin{proof}
The case $2\le p <\infty$ is given by Theorem \ref{thm-3.1}, while 
the case $1<p< 2$ for $D_p(\e, \eta)$ follows from the fact $D_p(\e, \eta)=D_{p^\prime}(\e, \eta)$.
Note that the $C^1$ conditions for $\Omega$ and $\{T_k\}$ are not needed for \eqref{D13} and \eqref{D14}.
\end{proof}

\begin{lemma}
Let $u\in W^{1, 2}_0(\Omega_{\e, \eta}) $ be a weak solution of the Dirichlet problem \eqref{D-01}, where
$F\in L^p(\Omega_{\e, \eta})$ and $f\in L^p(\Omega_{\e, \eta}; \R^d)$ for some $2< p< \infty$.
Then
\begin{equation}\label{L-O}
\| \nabla u\|_{L^p(\Omega_{\e, \eta})}
\le C \left\{
(\e\eta)^{-1} \| u\|_{L^p(\Omega_{\e, \eta})}
+ \| F \|_{L^p(\Omega_{\e, \eta})}
+ \| f\|_{L^p(\Omega_{\e, \eta})} \right\},
\end{equation}
where $C$ depends on $d$, $p$, $\Omega$ and $\{T_k\}$.
\end{lemma}

\begin{proof}
We claim that 
for any $x_0 \in \Omega_{\e, \eta}$,
\begin{equation}\label{L-01}
\int_{B(x_0, c_1\e \eta)\cap \Omega_{\e, \eta}}  |\nabla u|^p\, dx
\le C \int_{B(x_0, 8 c_1 \e \eta) \cap \Omega_{\e, \eta}}
\left( (\e\eta)^{-p} | u|^p + |F|^p + |f|^p \right)\, dx,
\end{equation}
where $c_1= (c_0/100)$.
To see this, 
we consider two cases.
In  the first case we assume $B(x_0, 2c_1\e \eta)\subset \Omega_{\e, \eta}$.
The estimate \eqref{L-01} then follows by the standard interior estimates for Laplace's equation.
In the second case we assume $B(x_0, 2c_1\e \eta)\cap \partial \Omega_{\e, \eta}\neq \emptyset$.
Choose $y_0\in B(x_0,2 c_1 \e \eta) \cap  \partial \Omega_{\e, \eta}$.
Then $B(x_0, c_1 \e \eta)\subset B(y_0, 3c_1 \e \eta)$.
Since $\partial\Omega$ is a $C^1$ domain and $\{T_k\}$ are the closures of bounded domains with uniform $C^1$ boundaries,
it follows from \eqref{Jk-F} by a localization argument that
\begin{equation}\label{L-02}
\int_{B(y_0, 3c_1\e \eta)\cap \Omega_{\e, \eta}}  |\nabla u|^p\, dx
\le C \int_{B(y_0, 6 c_1 \e \eta) \cap \Omega_{\e, \eta}}
\left( (\e\eta)^{-p} | u|^p + |F|^p + |f|^p \right)\, dx, 
\end{equation}
where we have used the fact $u=0$ on $\partial \Omega_{\e, \eta}$.
Note that  $B(y_0, 6c_1 \e \eta)\subset B(x_0, 8c_1\e \eta)$.
As a result, \eqref{L-01} follows from \eqref{L-02}.
Finally, we obtain \eqref{L-O} by integrating both sides of  \eqref{L-01} in $x_0$ over $\Omega_{\e, \eta}$
and using Fubini's Theorem.
\end{proof}

\begin{thm}\label{thm-D13}
Let $B_p(\e, \eta)$ be defined by \eqref{B-p}, but with $\Omega_{\e, \eta}$ in the place of $\omega_{\e, \eta}$.
Then

\begin{enumerate}

\item

For $1< p\le 2$,
\begin{equation}\label{D15}
B_p (\e, \eta) \le
\left\{ 
\aligned
 & C \min (1, \e \eta^{\frac{2-d}{2}}) & \quad &\text{ if } d\ge 3,\\
 & C \min (1, \e |\ln (\eta/2)|^{1/2}) & \quad & \text{ if } d=2,
 \endaligned
 \right.
 \end{equation}
 where $C$ depends on $d$, $p$, $c_0$ and $\Omega$.
 
 \item
 
 For $2< p< \infty$ and $d\ge 2$,
 \begin{equation}\label{D16}
 B_p (\e, \eta)\le
  C_\delta \, \e\eta
 \big\{ \min ( (\e \eta)^{-1}, \eta^{-\frac{d}{2}}) \big\}^{2-\frac{2}{p}+\delta}
 \end{equation}
 for any $\delta\in  (0, 1)$, where $C_\delta$ depends on $d$, $p$, $\delta$, $\{T_k\}$ and $\Omega$.
\end{enumerate}
\end{thm}

\begin{proof}
Since $B_p(\e, \eta)= C_{p^\prime} (\e, \eta)$ by Lemma \ref{lemma-D11},
the estimate \eqref{D15} follows from \eqref{D13}.
To prove \eqref{D16}, we first consider the case $d\ge 3$.
Let $u\in W^{1, 2}_0(\Omega_{\e, \eta})$ be a weak solution of
$-\Delta u=F$ in $\Omega_{\e, \eta}$, where $F\in L^p(\Omega_{\e, \eta})$.
By \eqref{D13} and \eqref{L-O} we obtain 
\begin{equation}\label{4.3-3}
\| \nabla u\|_{L^p(\Omega_{\e, \eta})}
\le C (\e \eta)^{-1} \big\{ \min (1, \e \eta^{\frac{2-d}{2}}) \big\}^2 \| F \|_{L^p(\Omega_{\e, \eta})}
\end{equation}
for any $p>2$.
By Riesz-Thorin Theorem, 
this, together with the $L^2$ estimate,
$$
\|\nabla u\|_{L^2(\Omega_{\e, \eta})}
\le C \min (1, \e \eta^{\frac{2-d}{2}}) \| F\|_{L^2(\Omega_{\e, \eta})},
$$
yields 
$$
\| \nabla u \|_{L^p(\Omega_{\e, \eta})}
\le C \e \eta
\big\{ \min ( (\e\eta)^{-1}, \eta^{-\frac{d}{2}} ) \big\}^{1+t}  \|F\|_{L^p(\Omega_{\e, \eta})},
$$
where
$$
t=\frac{\frac{1}{2} -\frac{1}{p}}{\frac12 -\frac{1}{p_0}}
$$
and $2<p< p_0$.
Since $t \to 1-\frac{2}{p}$ as $p_0 \to \infty$, by choosing $p_0$ sufficiently large, we obtain 
$$
\| \nabla u \|_{L^p(\Omega_{\e, \eta})}
\le C_\delta\, \e\,  \eta 
\big\{ \min ( (\e\eta)^{-1}, \eta^{-\frac{d}{2}} ) \big\}^{2- \frac{2}{p} +\delta} \| F\|_{L^p(\Omega_{\e, \eta})},
$$
for any $\delta\in (0, 1)$.
The proof for the case $d=2$ is similar (the factor $\ln (\eta/2)$ is absorbed by $\eta^{-\delta}$).
\end{proof}

\begin{thm}\label{thm-apd}
Let $A_p(\e, \eta)$ be defined by \eqref{A-p}, but with $\Omega_{\e, \eta}$ in the place of $\omega_{\e,\eta}$.
Then
\begin{equation}\label{D18}
A_p (\e, \eta)
\le 
 C_\delta   \big\{ \min ( (\e\eta)^{-2}, \eta^{-d} ) \big\}^{ |\frac12 -\frac{1}{p}| + \delta}
\end{equation}
for $d\ge 2$ and $1< p< \infty$, where $C_\delta$ depends on $d$, $p$, $\delta$, $\Omega$ and $\{T_k\}$.
\end{thm}

\begin{proof}
Consider the case $d\ge 3$ and $2< p< \infty$.
Let $u\in W^{1, 2}_0 (\Omega_{\e, \eta})$ be a weak solution of
$-\Delta u =\text{\rm div}(f)$ in $\Omega_{\e, \eta}$, where
$f\in L^p(\Omega_{\e, \eta}; \R^d)$.
It follows from \eqref{L-O} and \eqref{D13} that 
$$
\|\nabla u \|_{L^p(\Omega_{\e, \eta})}
\le C (\e \eta)^{-1} \min (1, \e \eta^{\frac{2-d}{2}}) \| f\|_{L^p(\Omega_{\e, \eta})}.
$$
By interpolation, this, together with the $L^2$ estimate $\| \nabla u \|_{L^2(\Omega_{\e, \eta})} \le \| f\|_{L^2(\Omega_{\e, \eta})}$,
yields
$$
\| \nabla u \|_{L^p (\Omega_{\e, \eta})}
\le C \big\{  \min ((\e \eta)^{-1} ,  \eta^{-\frac{d}{2}} ) \big\}^t \| f\|_{L^p(\Omega_{\e, \eta})},
$$
where $t$ is the same as in the proof of Theorem \ref{thm-D13} and $2<p<p_0$.
Since $t \to 1-\frac{2}{p}$ as $p_0 \to \infty$, we obtain \eqref{D18} for $d\ge 3$.
The proof for the case $d=2$ is similar. 
\end{proof}

\begin{proof}[Proof of Theorem \ref{main-thm-3}]
This follows from Lemma \ref{lemma-b} with estimates for  $A_p(\e, \eta)$, $B_p(\e, \eta)$, and $D_p(\e, \eta)$ given by  Theorem \ref{thm-apd},
\ref{thm-D13} and \ref{thm-D12}, respectively.
The estimate for $C_p (\e, \eta)$ follows from Theorem \ref{thm-D13} and  the fact $C_p(\e, \eta)=B_{p^\prime} (\e, \eta)$.
\end{proof}


 \bibliographystyle{amsplain}
 
\bibliography{W-1-p.bbl}

\bigskip

\begin{flushleft}

Zhongwei Shen,
Department of Mathematics,
University of Kentucky,
Lexington, Kentucky 40506,
USA.
E-mail: zshen2@uky.edu
\end{flushleft}

\bigskip

\end{document}